\documentclass[11pt]{article}
\usepackage{amsmath}
\usepackage{amssymb}
\usepackage{amsfonts}
\usepackage[all,2cell]{xy} \UseAllTwocells \SilentMatrices
\usepackage{geometry,latexsym,amsmath,amstext,mathrsfs,amssymb,amsthm,amscd,amsopn,verbatim,graphics}
\newtheorem{defi}{Definition}[section]
\newtheorem{lm}{Lemma}[section]
\newtheorem{thm}{Theorem}[section]
\newtheorem{prop}{Proposition}[section]

\newtheorem{rem}{Remark}[section]
\newtheorem{ex}{Example}[section]
\newcommand{\tensor}{\otimes}
\newcommand{\maps}{\colon}
\newcommand{\Cat}{{\rm Cat}}
\newcommand{\Set}{{\rm Set}}

\newcommand{\A}{{\mathcal A}}
\newcommand{\B}{{\mathcal B}}

\newcommand{\N}{{\mathcal N}}
\newcommand{\Pe}{{\mathcal P}}
\newcommand{\R}{{\mathcal R}}
\newcommand{\Sss}{{\mathcal Ss}}
\newcommand{\G}{{\mathcal G}}

\newcommand{\E}{{\mathcal E}}
\newcommand{\Ha}{{\mathcal H}}

\newcommand{\C}{{\mathcal C}}
\newcommand{\D}{{\mathcal D}}

\newcommand{\Q}{{\mathcal Q}}
\newcommand{\Ss}{{\mathcal S}}
\newcommand{\Fol}{{\mathcal Fol}}

\newcommand{\tri}{\triangleright}
\newcommand{\tril}{\triangleleft}

\numberwithin{equation}{section}
\begin{document}

\title{The simplicial interpretation of bigroupoid 2-torsors}
\author{Igor Bakovi\' c \\ Rudjer Bo\v skovi\' c Institute\\Theoretical Physics Department}
\date{}

\maketitle
\begin{abstract}
Actions of bicategories arise as categorification of actions of
categories. They appear in a variety of different contexts in
mathematics, from Moerdijk's classification of regular Lie groupoids
in foliation theory \cite{Mo} to Waldmann's work on deformation
quantization \cite{W}. For any such action we introduce an {\it
action bicategory}, together with a canonical projection (strict)
2-functor to the bicategory which acts. When the bicategory is a
bigroupoid, we can impose the additional condition that action is
principal in bicategorical sense, giving rise to a {\it bigroupoid
2-torsor}. In that case, the Duskin nerve of the canonical
projection is precisely the Duskin-Glenn simplicial 2-torsor,
introduced in \cite{Gl}.
\end{abstract}
\thispagestyle{empty}

\footnotetext[1]{This research was in part supported by the Croatian
Ministry of Science, Education and Sport, Project No.
098-0982930-2990.}

\footnotetext[2]{The author acknowledges support from European
Commission under the FP6 MRTN-CT-2006-035505 HEPTOOLS Marie Curie
Research Training Network}

\footnotetext[2]{The author acknowledge support from the project
"Applications of nonabelian cohomology to Geometry, Algebra and
Physics", Fonds DAAD06-346}

\footnotetext[3]{Author's email address: ibakovic@gmail.com}
\newpage
\section{Introduction}
There are several different ways to characterize those simplicial
sets which arise as nerves of categories, and the most of this
(equivalent) ways rely on the Quillen closed model structure on the
category $\Ss Set$ of simplicial sets. Simplicial sets which are
fibrant objects for the closed model structure on $\Ss Set$ are
called Kan complexes, and they are characterized by certain {\it
horn filling} conditions describing their exactness properties. This
conditions for a simplicial set $X_{\bullet}$ explicitly use a {\it
simplicial kernel} $K_{n}(X_{\bullet})$ in dimension $n$
\[K_{n}(X_{\bullet})=\{(x_{0},x_{1},\ldots,x_{i},\ldots,x_{j},\ldots,x_{n-1},x_{n})|
d_{i}(x_{j})=d_{j-1}(x_{i}), i<j \} \subseteq X_{n-1}^{n+1}\] which
is interpreted as the set of all possible sequences of
(n-1)-simplices which could possibly be the boundary of any
n-simplex. There exists a natural {\it boundary map}
\begin{equation} \label{bound}
\xymatrix{\partial_{n} \maps X_{n} \to K_{n}(X_{\bullet})}
\end{equation}
which takes any n-simplex $x \in X_{n}$ to the sequence
$\partial_{n}(x)=(d_{0}(x),d_{1}(x),\ldots,d_{n-1}(x),d_{n}(x))$ of
its (n-1)-faces. The set $\bigwedge^{k}_{n}(X_{\bullet})$ of k-horns
in dimension $n$
\[ \bigwedge^{k}_{n}(X_{\bullet})=\{(x_{0},x_{1},\ldots,x_{k-1},x_{k+1},\ldots,x_{n-1},x_{n})|
d_{i}(x_{j})=d_{j-1}(x_{i}), i<j, i,j\neq k\} \subseteq
X_{n-1}^{n}\] is the set of all possible sequences of
(n-1)-simplices which could possibly be the boundary of any
n-simplex, except that we $k^{th}$ face is missing. The k-horn map
in dimension $n$
\begin{equation} \label{horn}
\xymatrix{p_{n}^{k}(x) \maps X_{n} \to
\bigwedge^{k}_{n}(X_{\bullet})}
\end{equation} is
defined by the composition of the boundary map (\ref{bound}), with
the natural projection $q_{n}^{k}(x) \maps K_{n}(X_{\bullet}) \to
\bigwedge^{k}_{n}(X_{\bullet})$, which just omits the $k^{th}$
(n-1)-simplex from the sequence. Then we say that for $X_{\bullet}$
the {\it $k^{th}$ Kan condition in dimension $n$} is satisfied
(exactly) if the k-horn map (\ref{horn}) is surjection (bijection).
If Kan conditions are satisfied for all $0 < k < n$ and for all $n$,
then we say that $X_{\bullet}$ is a {\it weak Kan complex}, and if
Kan conditions are satisfied for extremal horns as well $0 \leq k
\leq n$ and for all $n$, then we say that $X_{\bullet}$ is a {\it
Kan complex}.

One of the above mentioned characterizations of nerves of
categories, first observed by Street, is that the simplicial set
$X_{\bullet}$ is the nerve of a category if and only if it is a weak
Kan complex in which the weak Kan conditions are satisfied exactly.
Weak Kan complexes were introduced by Boardman and Vogt \cite{BV} in
their work on homotopy invariant algebraic structures. These objects
are fundamental in the recent work of Joyal \cite{Jo}, which is so
far the most advanced form of the interplay between the category
theory and the simplicial theory. He even used the name {\it
quasicategory}, instead of the weak Kan complex, in order to
emphasize that {\it "most concepts and results of category theory
can be extended to quasicategories"}.

Similar characterization of nerves of groupoids leads to the
fundamental simplicial objects introduced by Duskin in \cite{Du2}.
An {\it n-dimensional Kan hypergroupoid}, is a Kan complex
$X_{\bullet}$ in which Kan conditions (\ref{horn}) are satisfied
exactly for all $m > n$ and $0 \leq k \leq m$. Glenn used the name
{\it n-dimensional hypergroupoid} in \cite{Gl} for any simplicial
set in which Kan conditions are satisfied exactly above dimension
$n$, while Beke called them in \cite{Be2} {\it exact n-types}, in
order to emphasize their homotopical meaning. These simplicial sets
morally play the role of nerves of weak n-groupoids, which is known
to be valid for small $n$. Consequently, a simplicial set
$X_{\bullet}$ is the nerve of a groupoid if and only if it is a
1-dimensional Kan hypergroupoid, and similar characterization holds
for nerves of bigroupoids.

Bigroupoids and bicategories, introduced by B\'enabou \cite{Be} in
1967, are weakest possible generalization of ordinary groupoids and
categories, respectively, to the immediate next level. In a
bicategory (bigroupoid), Hom-sets become categories (groupoids) and
the composition becomes functorial instead of functional. This
changes properties of associativity and identities which only hold
up to {\it coherent natural isomorphisms}. The {\it coherence laws}
which this natural isomorphisms satisfy, are the deep consequence of
the process called {\it categorification}, invented by Crane
\cite{Cr}, \cite{CF}, in which we find category theoretic analogs of
set theoretic concepts by replacing sets with categories, equations
between elements of the sets by isomorphisms between objects of the
category, functions by functors and equations between functions by
natural isomorphisms between functors.

The categorification become an essential tool in many areas of
modern mathematics. By generalizing algebraic concepts from the
classical set theory to the context of higher category theory, Baez
developed a program \cite{BD} of {\it higher dimensional algebra} in
an attempt to unify quantum field theory with traditional algebraic
topology. Later, Baez and Schreiber developed a {\it higher gauge
theory} \cite{BS1}, \cite{BS2} which describes the parallel
transport of strings using 2-connections on principal 2-bundles, as
the categorification of the usual gauge theory which describes the
parallel transport of point particles using connections on principal
bundles. Vector 2-spaces arose as a categorification of vector
spaces in the work of Kapranov and Voevodsky \cite{KV}, and they
were used by Baas, Dundas and Rognes \cite{BDR}, who defined vector
2-bundles in a search for a geometrically defined elliptic
cohomology. Later, Baas, B\"okstedt and Kro used topological
bicategories and vector 2-bundles \cite{BBK} in order to develop
2-categorical K-theory as the categorification of the usual
K-theory.

Another essential tool which we used is an {\it internalization}.
This is a process of generalizing concepts from the category $Set$
of sets, which are described in terms of sets, functions and
commutative diagrams, to concepts in another category $\E$ by
describing them in terms of objects, morphisms, and commutative
diagrams in $\E$. The internalization of the particular algebraic or
geometric structure in the category $\E$ rely on exactness
properties of $\E$ needed to describe corresponding commutative
diagrams. Therefore, the choice of the category $\E$ will depend on
the algebraic or geometric structure one wants to describe.

The most natural choice for an internalization and a
categorification of algebraic and geometric structures is a {\it
topos}, which is according to Grothendieck, the ultimate
generalization of the concept of space.

Let us now describe the content and the main results of the paper.

In Chapter 2 we recall some basic simplicial methods which we will
extensively use in the thesis. Most of this material is standard and
can be found in a classical book \cite{May} by May, or in a modern
treatment in \cite{GJ}. However, we also recall some more exotic
endofunctors on a category $\Ss Set$ of simplicial sets, such as the
{\it n-Coskeleton} $Cosk^{n}$ and the {\it shift functor} or {\it
d\'ecalage} $Dec$ which can be find in \cite{Du1}. Actions and
n-torsors over n-dimensional Kan hypergroupoids are defined by Glenn
in \cite{Gl} using simplicial maps which we call {\it exact
fibrations}. A simplicial map $\lambda_{\bullet} \maps \E_{\bullet}
\to \B_{\bullet}$ is an exact fibration in dimension $n$, if for all
$0 \leq k \leq n$, the diagrams
\[\xymatrix@!=4pc{E_{n} \ar[d]_{p_{\bar{k}}} \ar[r]^-{\lambda_{n}} &
B_{n} \ar[d]^{p_{\bar{k}}} \\
\bigwedge^{k}_{n}(\E_{\bullet}) \ar[r]^{} &
\bigwedge^{k}_{n}(\B_{\bullet})}\] are pullbacks. It is called an
exact fibration if it is an exact fibration in all dimensions. At
the end of this chapter, we describe two crucial concepts from
\cite{Gl} which we will use later in the thesis. An action of the
n-dimensional hypergroupoid $\B_{\bullet}$ is given in Definition
2.13 as a simplicial map $\lambda_{\bullet} \maps \Pe_{\bullet} \to
\B_{\bullet}$ which is an exact fibration for all $m \geq n$, and an
n-dimensional hypergroupoid n-torsor over $X$ in $\E$ is given in
Definition 2.14 as a simplicial map $\lambda_{\bullet} \maps
\Pe_{\bullet} \to \B_{\bullet}$ such that $P_{\bullet}$ is augmented
over $X$, aspherical and $n-1$-coskeletal.

In Chapter 3, definitions of a bicategory, their homomorphisms,
pseudonatural transformations and modifications are given as they
were defined by B\'enabou in his classical paper \cite{Be}. Then
Chapter 4 describes the Duskin nerve for bicategories as a geometric
nerve defined by the singular functor of the fully faithful
embedding
\begin{equation} \label{skelbicat}
\xymatrix{i \maps \Delta \to Bicat}
\end{equation}
of the skeletal simplicial category $\Delta$ into the category
$Bicat$ of bicategories and strictly unital homomorphism of
bicategories, constructed by B\'enabou in \cite{Be}. This embedding
regards any ordinal $[n]$ as the locally discrete 2-category, in the
sense that Hom-categories are discrete, so there exist only trivial
2-cells. We show that the Duskin nerve functor
\begin{equation} \label{nerbicat}
\xymatrix{N_{2} \maps Bicat \to \Ss Set}
\end{equation}
is fully faithful in Theorem 4.1 based on the result that the
geometric nerve provides a fully faithful functor on the category
$2-Cat_{lax}$ of 2-categories and normal lax 2-functors given in
\cite{BBF1}. The sets of n-simplices of the nerve $N_{2}\B$ of a
bicategory $\B$ are defined by $Hom_{Bicat}(i[n],\B)$, which were
explicitly described by Duskin \cite{Du5} in a geometric form.

In Chapter 5, we introduce {\it the second new concept} of this
paper, {\it action of a bicategory}, in Definition 5.1 as a
categorification of an action of a category. For an internal
bicategory $\B$ given by a bigraph in a finitely complete category
$\E$, and an internal category $\Pe$
\begin{equation}
\begin{array}{c}\label{actionbicat}
\xymatrix@!=2pc{P_{1} \ar@<1ex>[d]^-{s} \ar@<-1ex>[d]_{t} & B_{2}
\ar@<1ex>[d]^{s_{1}} \ar@<-1ex>[d]_{t_{1}}\\ P_{1}
\ar[dr]_{\Lambda_{0}} & B_{1}
\ar@<1ex>[d]^{s_{0}} \ar@<-1ex>[d]_{t_{0}}\\
& B_{0}}
\end{array}
\end{equation}
together with the {\it momentum functor} $\Lambda \maps \Pe \to
\B_{0}$ to a discrete category $\B_{0}$ of objects of the bicategory
$\B$, an {\it action functor}
\begin{equation} \label{Action}
\xymatrix{A \maps \Pe \times_{\B_{0}} \B_{1} \to \Pe}
\end{equation} is a categorification of an action of the category.
We introduce coherence laws for this action, which express the fact
that categories with an action of the bicategory $\B$ are {\it
pseudoalgebras} over a {\it pseudomonad} \cite{Her2}, \cite{L1},
\cite{LMV} naturally defined by $\B$. We give a description of an
Eilenberg-Moore 2-category of actions of the bicategory $\B$,
without details of the construction for corresponding pseudoalgebras
over a pseudomonad. In Chapter 6, for each action
(\ref{actionbicat}) of a bicategory $\B$ on a category $\Pe$, we
define {\it the third new concept}, an {\it action bicategory} $\Pe
\tril \B$ whose construction is given in Theorem 6.1. Then we see in
Proposition 6.1 that an action bicategory $\Pe \tril \B$ comes with
a canonical projection
\begin{equation} \label{homobicat}
\xymatrix{\Lambda \maps \Pe \tril \B \to \B}
\end{equation}
to the bicategory $\B$, which is a strict homomorphism of
bicategories.

Finally, in Chapter 7 we define {\it the fourth new concept}, and
our main geometric object - a {\it bigroupoid 2-torsor}. In
Definition 7.2 we define a bigroupoid 2-torsor as a bundle of
groupoids $\pi \maps \Pe \to X$ over an object $X$ in the category
$\E$, for which the induced functor
\begin{equation} \label{2torscon}
\xymatrix{(Pr_{1},A) \maps \Pe \times_{\B_{0}} B_{1} \to \Pe
\times_{X} \Pe}
\end{equation}
for an action (\ref{actionbicat}) is a strong equivalence of
groupoids. The first main result of the paper is Theorem 7.1 in
Chapter 7 which proves that for an action (\ref{actionbicat}) of an
internal bigroupoid $\B$ on groupoid $\Pe$, the simplicial map
$\Lambda_{\bullet}=N_{2}(\Lambda) \maps \Q_{\bullet} \to
\B_{\bullet}$ which arise as an application of a Duskin nerve for
bicategories (\ref{Action}) on a canonical homomorphism of
bicategories (\ref{homobicat}) is a (simplicial) action of the
bigroupoid $\B$ on the groupoid $\Pe$, i.e. it is an exact fibration
for all $n \geq 2$. The second main result of the paper is Theorem
7.2 which proves that for any $\B$-2-torsor $\Pe$ over $X$, the
simplicial map $\Lambda_{\bullet}=N_{2}(\Lambda) \maps \Q_{\bullet}
\to \B_{\bullet}$ is a Glenn's 2-torsor, which is an internal
simplicial map $\Lambda_{\bullet} \maps P_{\bullet} \to
\B_{\bullet}$ in $\Ss(\E)$, which is an exact fibration for all $n
\geq 2$, and where $P_{\bullet}$ is augmented over $X$, aspherical
and 1-coskeletal ($P_{\bullet} \simeq Cosk^{1}(P_{\bullet}))$.

\section{Simplicial objects}
In this section we will review some standard notions from the theory
of simplicial sets. Most of the statements and proofs may be found
in standard textbooks \cite{GJ} or \cite{May}.

\begin{defi} Skeletal simplicial category $\Delta$ consists of the
following data:
\begin{itemize}
\item objects are finite nonempty ordinals $[n]=\{0<1<...<n\}$,

\item morphisms are monotone maps $f \maps [n] \to [m]$, which for
all $i,j \in [n]$ such that $i \leq j$, satisfy $f(i) \leq f(j)$.
\end{itemize}
We also call $\Delta$ the topologist's simplicial category, and this
is a full subcategory of the algebraist's simplicial category
$\bar{\Delta}$, which has an additional object $[-1]=\emptyset$,
given by a zero ordinal, that is an empty set.
\end{defi}

Skeletal simplicial category $\Delta$ may be also given by means of
generators given by the diagram

\[\xymatrix{[0] \ar@<-0.5ex>[rr]_{\partial_{0}} \ar@<0.5ex>[rr]^{\partial_{1}}  &&
[1] \ar[ll] \ar@<-1ex>[rr]_{\partial_{0}}
\ar@<1ex>[rr]^{\partial_{2}} \ar[rr] && [2] \ar@<-0.5ex>[rr]
\ar@<0.5ex>[rr] \ar@<-1.5ex>[rr]_{\partial_{0}}
\ar@<1.5ex>[rr]^{\partial_{3}} \ar@<-0.5ex>[ll] \ar@<0.5ex>[ll] &&
[3] \ar[ll] \ar@<-1ex>[ll] \ar@<1ex>[ll]...  }\] and relations given
by the maps $\partial_{i} \maps [n-1] \to [n]$ for $0 \leq i \leq
n-1$, called coface maps, which are injective maps that omit $i$ in
the image, and the maps $\sigma_{i} \maps [n] \to [n-1]$ for $0 \leq
i \leq n-1$, called codegeneracy maps, which are surjective maps
which repeat $i$ in the image. These maps satisfy following
cosimplicial identities:
\[\begin{array}{c}
\partial_{j}\partial_{i}=\partial_{i}\partial_{j-1} \hspace{2cm} (i<j)\\
\sigma_{j}\sigma_{i}=\sigma_{i}\sigma_{j+1} \hspace{2cm} (i \leq j)\\
\sigma_{j}\partial_{i}=\partial_{i}\sigma_{j-1} \hspace{2cm} (i<j)\\
\hspace{0.9cm} \sigma_{j}\partial_{i}=id \hspace{2cm} (i=j,i=j+1)\\
\hspace{0.7cm} \sigma_{j}\partial_{i}=\partial_{i}\sigma_{j+1}
\hspace{2cm} (i > j+1)
\end{array}\]

We will use the following factorization of monotone maps by means of
cofaces and codegeneracies.

\begin{lm} Any monotone map $f \maps [m] \to [n]$ has a unique
factorization given by
\[\xymatrix{f=\partial^{n}_{i_1}\partial^{n-1}_{i_2}...\partial^{n-s+1}_{i_s}
\sigma^{m-t}_{j_t}...\sigma^{m-2}_{j_2}\sigma^{m-1}_{j_1}}\] where
$0 \leq i_s < i_{s-1} < ... < i_1 \leq n$, $0 \leq j_t < j_{t-1} <
... < j_1 \leq m$ and $n=m-t+s$.
\begin{proof} The proof follows directly from the injective-surjective factorization in $\Set$
and simplicial identities.
\end{proof}
\end{lm}

\begin{defi} Simplicial object $X_{\bullet}$ in a category $\C$ is a functor $X \maps
\Delta^{op} \to \C$. This is an object of the category $\Ss(\C)$
whose morphisms are natural transformations, which we call internal
simplicial morphisms. In the case when the category $\C=\Set$ is the
category of sets (in a fixed Grothendieck universe), then we call
$X_{\bullet}$ a simplicial set, and we denote the corresponding
category of simplicial sets by $\Ss Set$.
\end{defi}
Thus we can view a simplicial object $X_{\bullet}$ in $\C$ as a
diagram
\[\xymatrix{X_{0} \ar[rr]  && X_{1} \ar@<-0.5ex>[ll]_{d_{1}} \ar@<0.5ex>[ll]^{d_{0}}
\ar@<-0.5ex>[rr] \ar@<0.5ex>[rr] && X_{2} \ar[ll]
\ar@<-1ex>[ll]_{d_{2}} \ar@<1ex>[ll]^{d_{0}} \ar[rr] \ar@<-1ex>[rr]
\ar@<1ex>[rr] &&X_{3}... \ar@<-0.5ex>[ll] \ar@<0.5ex>[ll]
\ar@<-1.5ex>[ll]_{d_{3}} \ar@<1.5ex>[ll]^{d_{0}} }\] in $\C$, where
we denoted just extremal face operators, and left the signature for
inner face operators, and degeneracies.

Then the following simplicial identities hold:
\[\begin{array}{c}
d_{i}d_{j}=d_{j-1}d_{i} \hspace{2cm} (i<j)\\
s_{i}s_{j}=s_{j+1}s_{i} \hspace{2cm} (i \leq j)\\
d_{i}s_{j}=s_{j-1}d_{i} \hspace{2cm} (i<j)\\
\hspace{0.9cm} d_{i}s_{j}=id \hspace{2cm} (i=j,i=j+1)\\
\hspace{0.7cm} d_{i}s_{j}=s_{j+1}d_{i} \hspace{2cm} (i > j+1)
\end{array}\] where $d_{i}:=X(\partial_{i})$ and
$s_{i}:=X(\sigma_{i})$.

\begin{defi} An augmented simplicial object $X_{\bullet} \to X_{-1}$
in a category $\C$ is a functor $X \maps \bar{\Delta}^{op} \to \C$.
This is an object of the category $\Ss_{a}(\C)$ whose morphisms are
natural transformations, which we call simplicial maps of augmented
simplicial objects.
\end{defi}

In order to define basic endofunctors on the category $\Ss(\C)$,
which we will use in the thesis, we first need to describe the
process of a truncation of internal simplicial objects. For any
natural number $n$, we have the full subcategory $\Delta_{n}$ of the
simplicial category $\Delta$, whose objects are the first $n+1$
ordinals. Then we have the following definition.

\begin{defi} Let $X_{\bullet}$ be a simplicial object in $\C$.
An n-truncated simplicial object $tr_{n}(X_{\bullet})$ in a category
$\C$ is a functor $X i_{n} \maps \Delta_{n}^{op} \to \C$ given by
the precomosition with an embedding $i_{n} \maps \Delta_{n} \to
\Delta$. This is an object of the category $\Ss^{n}(\C)$,  and we
have an n-truncation functor
\[ tr^{n} \maps \Ss(\C) \to \Ss^{n}(\C) \] from the category
$\Sss(\C)$ of simplicial objects in $\C$, to the category
$\Sss^{n}(\C)$ of n-truncated simplicial objects in $\C$.
\end{defi}

If $\C$ is a finitely complete category, an n-truncation functor
$tr^{n} \maps \Ss(\C) \to \Ss^{n}(\C)$ has a right adjoint $cosk^{n}
\maps \Ss^{n}(\C) \to \Ss(\C)$, and if $\C$ is a finitely cocomplete
category, it has a left adjoint $sk^{n} \maps \Ss^{n}(\C) \to
\Ss(\C)$.

The corresponding comonad $Sk^{n}=sk^{n}tr^{n} \maps \Ss Set \to \Ss
Set$ for $\C=\Set$ is easy to describe. For any simplicial set
$X_{\bullet}$, its skeleton $Sk^{n}(X_{\bullet})$ is a simplicial
subset of $X_{\bullet}$, which is identical to $X_{\bullet}$ in all
dimensions $k \leq n$, and has only degenerate simplices in all
higher dimensions.

The monad $Cosk^{n}=cosk^{n}tr^{n} \maps \Ss(\C) \to \Ss(\C)$ is
described by the simplicial kernel.

\begin{defi} The $n^{th}$ simplicial kernel of the simplicial object
$X_{\bullet}$ is an object $K_{n}(X_{\bullet})$ in $\C$, together
with morphisms $pr_{j} \maps K_{n}(X_{\bullet}) \to X_{n-1}$ for
$j=0,\ldots,n$, which is universal with respect to relations
$d_{i}pr_{j}=pr_{j-1}d_{i}$, for all $0 \leq i<j \leq n$.
\end{defi}

Now, let we describe in more detail the monad
$Cosk^{n}=cosk^{n}tr^{n} \maps \Ss Set \to \Ss Set$ in the case
$\C=\Set$, that is when we deal with simplicial sets.

The simplicial kernel of the simplicial set $X_{\bullet}$ in
dimension $n$ is a set $K_{n}(X_{\bullet})$ defined by
\[K_{n}(X_{\bullet})=\{(x_{0},x_{1},\ldots,x_{i},\ldots,x_{j},\ldots,x_{n-1},x_{n})|
d_{i}(x_{j})=d_{j-1}(x_{i}), i<j \} \subseteq X_{n-1}^{n+1}\] so
that we can interpret it as the set of all possible sequences of
(n-1)-simplices which could possibly be the boundary of any
n-simplex. If $x \in X_{n}$ is an n-simplex in a simplicial set
$X_{\bullet}$, its boundary $\partial_{n}(x)$ is a sequence of its
(n-1)-faces
\[\partial_{n}(x)=(d_{0}(x),d_{1}(x),\ldots,d_{n-1}(x),d_{n}(x)).\]

Then, for the simplicial set $X_{\bullet}$, the simplicial set
$Cosk^{n}(X_{\bullet})$ is identical to $X_{\bullet}$ in all
dimensions $k \leq n$, and the set of (n+1)-simplices of
$Cosk^{n}(X_{\bullet})$ is defined by
\[Cosk^{n}(X_{\bullet})_{n+1}=K_{n+1}(X_{\bullet})\] while the face
operators are given by the projections $d_{i}=pr_{i} \maps
K_{n+1}(X_{\bullet}) \to X_{n}$ for all $0 \leq i \leq n+1$. All of
the higher dimensional set of simplices of $Cosk^{n}(X_{\bullet})$
are obtained just by inductively iterating the simplicial kernels
\[Cosk^{n}(X_{\bullet})_{n+2}=K_{n+2}(tr^{n+1}Cosk^{n}(X_{\bullet}))\] and
so on.

From the universal property of the $n^{th}$ simplicial kernel
$K_{n}(X_{\bullet})$, we have a canonical morphism
$\delta_{n}=(d_{0},d_{1},\ldots,d_{n-1},d_{n}) \maps X_{n} \to
K_{n}(X_{\bullet})$, called the {\it boundary} of the object of
n-simplices, or briefly the $n^{th}$ boundary morphism.

The first nontrivial component of the unit $\eta \maps Id_{\Sss} \to
Cosk^{n}$ of the adjunction is given by $(n+1)^{th}$ boundary
morphism
\[\delta_{n+1}=(d_{0},d_{1},\ldots,d_{n},d_{n+1}) \maps
X_{n+1} \to Cosk^{n}(X_{\bullet})_{n+1}=K_{n+1}(X_{\bullet})\] and
we have following definitions.

\begin{defi} We say that the simplicial object $X_{\bullet}$ in $\C$ is
coskeletal in dimension $n$, or n-coskeletal, if the unit $\eta
\maps Id_{\Ss Set} \to Cosk^{n}$ of the adjunction is a natural
isomorphism. Similarly, we say that the simplicial object
$X_{\bullet}$ in $\C$ is skeletal in dimension $n$, or n-skeletal,
if the counit $\epsilon \maps Sk^{n} \to Id_{\Ss Set}$ of the
adjunction is a natural isomorphism.
\end{defi}

\begin{defi} We say that the simplicial object $X_{\bullet}$ in $\C$ is
aspherical in dimension $n$ if the $n^{th}$ boundary morphism
$\delta_{n} \maps X_{n} \to K_{n}(X_{\bullet})$ is an epimorphism.
If $X_{\bullet}$ is aspherical in all dimensions, then we say that
it is aspherical.
\end{defi}

In order to define Kan complexes later, we use another universal
construction which formally describe `hollow' simplices, or
simplices in which the $k^{th}$ face is missing.

\begin{defi} The k-horn in dimension $n$ of the simplicial object
$X_{\bullet}$ is an object $\bigwedge^{k}_{n}(X_{\bullet})$ in $\C$,
together with morphisms $p_{i} \maps \bigwedge^{k}_{n}(X_{\bullet})
\to X_{n-1}$ for $i=0,\ldots,n$ and $i \neq k$, which is universal
with respect to relations $d_{i}p_{j}=p_{j-1}d_{i}$, for all $0 \leq
i<j \leq n$ and $i,j \neq k$.
\end{defi}

The set $\bigwedge^{k}_{n}(X_{\bullet})$ of k-horns in dimension $n$
\[ \bigwedge^{k}_{n}(X_{\bullet})=\{(x_{0},x_{1},\ldots,x_{k-1},x_{k+1},\ldots,x_{n-1},x_{n})|
d_{i}(x_{j})=d_{j-1}(x_{i}), i<j, i,j\neq k\} \subseteq
X_{n-1}^{n}\] is the set of all possible sequences of
(n-1)-simplices which could possibly be the boundary of any
n-simplex, except that we $k^{th}$ face is missing. Then for the
simplicial set $X_{\bullet}$, the k-horn map in dimension $n$
\[p_{n}^{k}(x) \maps X_{n} \to \bigwedge^{k}_{n}(X_{\bullet})\] is
defined by the composition of the boundary map $\partial_{n} \maps
X_{n} \to K_{n}(X_{\bullet})$, with the projection $q_{n}^{k}(x)
\maps K_{n}(X_{\bullet}) \to \bigwedge^{k}_{n}(X_{\bullet})$, and it
just omits the $k^{th}$ (n-1)-simplex from the sequence.

If $x \in X_{n}$ is an n-simplex, its k-horn $p_{n}^{k}(x)$ is
defined by the image of the projection of its boundary to the
sequence of faces in which the $k^{th}$ face is omitted
\[p_{n}^{k}(x)=(d_{0}(x),d_{1}(x),\ldots,d_{k-1}(x),d_{k+1}(x),\ldots,d_{n-1}(x),d_{n}(x))\]

Let $(x_{0},x_{1},\ldots,x_{k-1},-,x_{k+1},\ldots,x_{n-1},x_{n}) \in
\bigwedge^{k}_{n}(X_{\bullet})$ be a k-horn in dimension $n$. If
there exists an n-simplex $x \in X_{n}$ such that
\[p_{n}^{k}(x)=(x_{0},x_{1},\ldots,x_{k-1},-,x_{k+1},\ldots,x_{n-1},x_{n})\]
then we say that n-simplex $x$ is a filler of the horn.

\begin{defi} Let $X_{\bullet}$ be an simplicial object in the category $\C$.
We say that the $k^{th}$ Kan condition in dimension $n$ is satisfied
for $X_{\bullet}$ if the k-horn morphism
\[p_{n}^{k}(x) \maps X_{n} \to \bigwedge^{k}_{n}(X_{\bullet})\] is
an epimorphism. The condition is satisfied exactly if the above
morphism is an isomorphism. If Kan conditions are satisfied for all
$0 < k < n$ and for all $n$, then we say that $X_{\bullet}$ is a
weak Kan complex. Finally, if Kan conditions are satisfied for
extremal horns as well $0 \leq k \leq n$ and for all $n$, then we
say that $X_{\bullet}$ is a Kan complex.
\end{defi}

This condition can be stated entirely in the topos theoretic context
by using the sieves

\[ \bigwedge^{k}[n] \hookrightarrow \overset{\bullet}{\Delta}[n] \hookrightarrow
\Delta[n]\] in $\Ss Set$, where $\Delta[n]$ is the {\it standard
n-simplex}, which is just the simplicial set represented by the
ordinal $[n]$. The simplicial set $\overset{\bullet}{\Delta}[n]$ is
the {\it boundary of the standard n-simplex} which is identical to
standard n-simplex in all dimensions bellow $n$, and has only
degenerate simplices in higher dimensions. It is defined by the
(n-1)-skeleton $\overset{\bullet}{\Delta}[n]=Sk^{n-1}(\Delta[n])$ of
the standard n-simplex. The simplicial set $\bigwedge^{k}[n]$ is the
{\it k-horn of the standard n-simplex}, which is identical to
$\overset{\bullet}{\Delta}[n]$ except that it is not generated by
the simplex $\delta_{k} \maps [n-1] \to [n]$.

Using the Yoneda lemma
\[Hom_{\Ss Set}(\Delta[n],X_{\bullet}) \simeq  X_{n}\] the $n^{th}$ Kan
condition says that for any simplicial map $\bar{x} \maps
\bigwedge^{k}[n] \to X_{\bullet}$, there exist a simplicial map $x
\maps \Delta[n] \to X_{\bullet}$ such that the diagram
\[\xymatrix@!=3pc{\bigwedge^{k}[n] \ar@{^{(}->}[d]
\ar[r]^-{\bar{x}} & X_{\bullet} \\
\Delta[n] \ar[ur]_{x} & }\] commutes.

\begin{rem}
The $n^{th}$ Kan condition is equivalent to the injectivity of the
simplicial set $X_{\bullet}$ with respect to monomorphisms
$\bigwedge^{k}[n] \hookrightarrow \Delta[n]$ for all $0 \leq k \leq
n$. In this terms, Kan complex $X_{\bullet}$ is a simplicial set
which is injective with respect to all monomorphisms
$\bigwedge^{k}[n] \hookrightarrow \Delta[n]$ for all $0 \leq k \leq
n$, and all $n \geq 0$.
\end{rem}

\begin{prop} Every aspherical simplicial object $X_{\bullet}$ is a Kan simplicial object.
\begin{proof} We will use the Barr embedding theorem and prove it in $\Set$.
Consider the diagram
\[\xymatrix@!=3pc{ & X_{n} \ar[dl]_{\delta_{n}}
\ar[dr]^-{p_{n}^{k}} & \\
K_{n}(X_{\bullet}) \ar[rr]_{q_{n}^{k}} &&
\bigwedge^{k}_{n}(X_{\bullet})}\] and a k-horn
$(x_{0},x_{1},\ldots,x_{k-1},-,x_{k+1},\ldots,x_{n},x_{n+1}) \in
\bigwedge^{k}_{n+1}(X_{\bullet})$. If there exists a filler $x \in
X_{n+1}$ for which
$p_{n+1}^{k}(x)=(x_{0},x_{1},\ldots,x_{k-1},-,x_{k+1},\ldots,x_{n},x_{n+1})$
then its k-face $d_{k}(x)=x_{k}$ has a boundary uniquely determined
by the simplices $x_{i}$ for $i \neq k$ since
\[d_{i}(x_{k})= \left\lbrace
\begin{array}{l} d_{k-1}(x_{i}) \hspace{2cm} 0 \leq i < k \leq n+1\\
d_{k}(x_{i+1}) \hspace{2cm} 0 \leq k \leq i \leq n+1
\end{array}\right. \] and therefore
$(d_{0}(x_{k}),d_{1}(x_{k}),\ldots,d_{n-1}(x_{k}),d_{n}(x_{k})) \in
K_{n}(X_{\bullet})$. Since we supposed that $\delta_{n} \maps X_{n}
\to K_{n}(X_{\bullet})$ is an epimorphism, then such a simplex
$x_{k} \in X_{n}$ really exists, and we conclude that the morphism
$q_{n+1}^{k} \maps K_{n+1}(X_{\bullet}) \to
\bigwedge^{k}_{n+1}(X_{\bullet})$ is also an epimorphism. But this
is true for all $n$, and it follows that $p_{n+1}^{k} \maps X_{n+1}
\to \bigwedge^{k}_{n+1}(X_{\bullet})$ is an epimorphism as a
composition of epimorphisms, and therefore $X_{\bullet}$ is a Kan
simplicial set.
\end{proof}
\end{prop}

\begin{rem} For any simplicial set $X_{\bullet}$ the simplicial
kernel $K_{1}(X_{\bullet})$ in dimension $1$ is equal to the product
$K_{1}(X_{\bullet})=X_{0} \times X_{0}$. For the augmented
simplicial set $X_{\bullet} \to X_{-1}$, when we have
$K_{1}(X_{\bullet})=X_{0} \times_{X_{-1}} X_{0}$. The set of k-horns
is given by $\bigwedge^{k}_{1}(X_{\bullet})=X_{0}$ for $k=0,1$, and
in each case maps $p_{1}^{k} \maps X_{1} \to
\bigwedge^{k}_{1}(X_{\bullet})$ and $q_{1}^{k} \maps
K_{1}(X_{\bullet}) \to \bigwedge^{k}_{1}(X_{\bullet})$ are always
epimorphisms.
\end{rem}

\begin{defi} A simplicial object $X_{\bullet}$ in $\C$ is said to be
split if there exist a family of morphisms $s_{n+1} \maps X_{n} \to
X_{n+1}$ for all $n \geq 0$, called the contraction for
$X_{\bullet}$, which satisfy all the simplicial identities involving
degeneracies. When a simplicial object is augmented $p \maps X_{0}
\to X_{-1}$ then the contraction includes also a morphism $s_{0}
\maps X_{-1} \to X_{0}$ such that $ps_{0}=id_{X_{-1}}$.
\end{defi}

\begin{rem} Any augmented split simplicial set $X_{\bullet} \to
X_{-1}$ may be seen as the simplicial set $X_{\bullet}$ together
with the homotopy equivalence $d_{\bullet} \maps X_{\bullet} \to
K(X_{-1},0)$ to the constant simplicial set $K(X_{-1},0)$ which has
$X_{-1}$ at each dimension and the identity maps for faces and
degeneracies. This means that there exists a simplicial map
$s_{\bullet} \maps K(X_{-1},0) \to X_{\bullet}$ such that the
compositions $s_{\bullet}d_{\bullet} \simeq id_{X_{\bullet}}$ and
$d_{\bullet}s_{\bullet} \simeq id_{K(X_{-1},0)}$ are homotopic to
respective identity simplicial maps.
\end{rem}
\begin{prop} Every augmented aspherical simplicial
set $X_{\bullet} \to X_{-1}$ is split.
\begin{proof} The proof follows by induction. Let's take any
section $s_{0} \maps X_{-1} \to X_{0}$ and we assume that we have
the $n^{th}$ contraction $s_{n} \maps X_{n-1} \to X_{n}$. Let
$q_{i}(x) \maps X_{n} \to K_{n+1}(X_{\bullet})$ be the $i^{th}$
degeneracy for the $n^{th}$ simplicial kernel of $X_{\bullet}$, and
we define $q_{n+1}(x) \maps X_{n} \to K_{n+1}(X_{\bullet})$ by
\[ q_{n+1}(x)=(s_{n}d_{0}(x),s_{n}d_{1}(x),\ldots,s_{n}d_{n-1}(x),s_{n}d_{n}(x)). \]
Now let's choose the splitting $s \maps K_{n+1}(X_{\bullet}) \to
X_{n+1}$ of the $(n+1)^{th}$ boundary map $\delta_{n+1}(x) \maps
X_{n+1} \to K_{n+1}(X_{\bullet})$, which is a surjection by
assumption, such that $s_{i}=s q_{i}$ for all $0 \leq i \leq n$.
Then the contraction $s_{n+1} \maps X_{n} \to X_{n+1}$ defined by
$s_{n+1}=s q_{n+1}$ satisfy all the identities involving
degeneracies since $q_{n+1}=\delta_{n+1}s
q_{n+1}=\delta_{n+1}s_{n+1}$.
\end{proof}
\end{prop}

An n-truncation functor has the extension to the {\it augmented
n-truncation functor}
\[tr^{n}_{a} \maps \Ss_{a}(\C) \to \Ss^{n}_{a}(\C)\] from the category
$\Ss_{a}(\C)$ of augmented simplicial objects in $\C$ to the
category $\Ss^{n}_{a}(\C)$ of n-truncated augmented simplicial
objects in $\C$. Since $\C$ is finitely complete, it has a right
adjoint $cosk^{n}_{a} \maps \Ss^{n}_{a}(\C) \to \Ss_{a}(\C)$, called
the {\it augmented n-coskeleton functor}. If we regard any augmented
simplicial object $X_{\bullet} \to X_{-1}$ in $\C$ as the ordinary
simplicial object in the slice category $(\C,X_{-1})$, then the
augmented n-coskeleton functor becomes ordinary n-coskeleton functor
in the slice category $(\C,X_{-1})$.

\begin{ex} The category $\C$ may be identified with the category
$\Ss^{-1}_{a}(\C)$ of -1-truncated augmented simplicial objects in
$\C$, and the augmented -1-truncation functor $tr^{-1}_{a} \maps
\Ss_{a}(\C) \to \Ss^{-1}_{a}(\C)$ assigns to any augmented
simplicial object $X_{\bullet} \to X_{-1}$ the object $X_{-1}$ of
$\C$. Its right adjoint is augmented -1-coskeleton functor
$cosk^{-1}_{a} \maps \Ss^{-1}_{a}(\C) \to \Ss_{a}(\C)$ which assigns
to any object $X$ in $\C$ the constant augmented simplicial object
\[\xymatrix{X && X \ar[rr] \ar[ll]_{id} && X_{1} \ar@<-0.5ex>[ll]_{id} \ar@<0.5ex>[ll]^{id}
\ar@<-0.5ex>[rr] \ar@<0.5ex>[rr] && X \ar[ll] \ar@<-1ex>[ll]_{id}
\ar@<1ex>[ll]^{id} \ar[rr] \ar@<-1ex>[rr] \ar@<1ex>[rr] && X \ldots
\ar@<-0.5ex>[ll] \ar@<0.5ex>[ll] \ar@<-1.5ex>[ll]_{id}
\ar@<1.5ex>[ll]^{id} }\] denoted by $K(X,0) \to X$.
\end{ex}

\begin{ex} The category of morphisms $\C^{I}$ of $\C$ may be identified with the category
$\Ss^{0}_{a}(\C)$ of 0-truncated augmented simplicial objects in
$\C$, and the augmented 0-truncation functor $tr^{0}_{a} \maps
\Ss_{a}(\C) \to \Ss^{0}_{a}(\C)$ assigns to any augmented simplicial
object $X_{\bullet} \to X_{-1}$ the morphism $d \maps X_{0} \to
X_{-1}$ of $\C$. Its right adjoint is augmented 0-coskeleton functor
$cosk^{0}_{a} \maps \Ss^{0}_{a}(\C) \to \Ss_{a}(\C)$ which assigns
to any morphism $d \maps X_{0} \to X_{-1}$ in $\C$ the simplicial
kernel of the morphism
\[\xymatrix{X_{-1} & X_{0} \ar[rr] \ar[l]_-{d} && X_{0} \times_{X_{-1}} X_{0}
\ar@<-0.5ex>[ll]_-{pr_{1}} \ar@<0.5ex>[ll]^-{pr_{2}}
\ar@<-0.5ex>[rr] \ar@<0.5ex>[rr] && X_{0} \times_{X_{-1}} X_{0}
\times_{X_{-1}} X_{0} \ar[ll] \ar@<-1ex>[ll]_-{pr_{12}}
\ar@<1ex>[ll]^-{pr_{23}} }\] denoted by $cosk^{0}_{a}(X_{0} \to
X_{-1})$.
\end{ex}

The corresponding monad and the comonad on the category
$\Ss_{a}(\C)$ of augmented simplicial objects in $\C$ are denoted by
$Cosk_{a} \maps \Ss_{a}(\C) \to \Ss_{a}(\C)$ and $Sk_{a} \maps
\Ss_{a}(\C) \to \Ss_{a}(\C)$ respectively, in accordance with the
case of nonaugmented simplicial objects in $\C$.

Another important construction on simplicial objects is given by the
so called {\it shift functor}. For any simplicial object
$X_{\bullet}$ in $\C)$, we restrict the corresponding functor $X
\maps \Delta^{op} \to \C$ to the subcategory of $\Delta^{op}$ with
the same objects, and with the same generators except for the
injections $\partial_{n} \maps [n-1] \to [n]$. If we renumber the
objects in $\Delta^{op}$, so that the ordinal $[n-1]$ becomes $[n]$,
we obtain a simplicial object in $\C$, denoted by
$Dec(X_{\bullet})$, which is augmented to the object $X_{0}$ (or to
the constant simplicial object $Sk^{0}(X_{\bullet})$ in $\C$) and is
contractible with respect to the simplicial map obtained from the
family $(s_{n})_{n \geq 0}$ of extremal degeneracies, as is shown in
the diagram

\begin{equation}
\begin{array}{l}\label{Decalage}
\xymatrix{X_{0} \ar@{=}[rr] \ar@<0.5ex>[d]^{s_{0}} && X_{0}
\ar@{=}[rr] \ar@<0.5ex>[d]^{s^{2}_{0}} && X_{0} \ar@{=}[rr]
\ar@<0.5ex>[d]^{s^{3}_{0}} &&
X_{0} \ar@<0.5ex>[d]^{s^{4}_{0}}&...& Sk^{0}(X_{\bullet}) \ar@<0.5ex>[d]^{S_{0}}\\
X_{1} \ar@<0.5ex>[u]^{d_{0}} \ar[rr] \ar@<0.5ex>[d]^{d_{1}} && X_{2}
\ar@<0.5ex>[u]^{d^{2}_{0}} \ar@<-0.5ex>[ll]_{d_{1}}
\ar@<0.5ex>[ll]^{d_{0}} \ar@<-0.5ex>[rr] \ar@<0.5ex>[rr]
\ar@<0.5ex>[d]^{d_{2}} && X_{3} \ar@<0.5ex>[u]^{d^{3}_{0}} \ar[ll]
\ar@<-1ex>[ll]_{d_{2}} \ar@<1ex>[ll]^{d_{0}} \ar@<0.5ex>[d]^{d_{3}}
\ar[rr] \ar@<-1ex>[rr] \ar@<1ex>[rr] && X_{4}
\ar@<0.5ex>[u]^{d^{4}_{0}} \ar@<0.5ex>[d]^{d_{4}} \ar@<-0.5ex>[ll]
\ar@<0.5ex>[ll] \ar@<-1.5ex>[ll] \ar@<1.5ex>[ll]
&...& Dec(X_{\bullet}) \ar@<0.5ex>[d]^{D_{1}} \ar@<0.5ex>[u]^{D_{0}} \\
X_{0} \ar@<0.5ex>[u]^{s_{0}} \ar[rr]  && X_{1}
\ar@<0.5ex>[u]^{s_{1}} \ar@<-0.5ex>[ll]_{d_{1}}
\ar@<0.5ex>[ll]^{d_{0}} \ar@<-0.5ex>[rr] \ar@<0.5ex>[rr] && X_{2}
\ar@<0.5ex>[u]^{s_{2}} \ar[ll] \ar@<-1ex>[ll]_{d_{2}}
\ar@<1ex>[ll]^{d_{0}} \ar[rr] \ar@<-1ex>[rr] \ar@<1ex>[rr] && X_{3}
\ar@<0.5ex>[u]^{s_{3}} \ar@<-0.5ex>[ll] \ar@<0.5ex>[ll]
\ar@<-1.5ex>[ll] \ar@<1.5ex>[ll] &...& X_{\bullet}
\ar@<0.5ex>[u]^{S_{1}}}
\end{array}
\end{equation}
where the simplicial map $S_{0} \maps Sk^{0}(X_{\bullet}) \to
Dec(X_{\bullet})$ on the right side of the diagram is defined by
$(S_{0})_{n}=(s_{0})^{n}=s_{0} s_{0} \ldots s_{0}$, and the
simplicial map $D_{0} \maps Dec(X_{\bullet}) \to
Sk^{0}(X_{\bullet})$ is defined by $(D_{0})_{n}=(d_{0})^{n}=d_{0}
d_{0} \ldots d_{0}$, for each level $n$. The other two simplicial
maps $S_{1} \maps X_{\bullet} \to Dec(X_{\bullet})$ and $D_{1} \maps
Dec(X_{\bullet}) \to X_{\bullet}$ are defined by $(S_{1})_{n}=s_{n}$
and $(D_{1})_{n}=d_{n}$ respectively.

The above construction extends to a functor
\[ Dec \maps \Ss(\C) \to \Ss_{as}(\C)\]
from the category of simplicial objects in $\C$, to the category
$\Sss_{as}(\C)$ of augmented split simplicial objects in $\C$. This
functor has a left adjoint, given by the forgetful functor
\[U \maps \Ss_{ac}(\C) \to \Ss(\C)\]
which forgets the augmentation and a splitting. Thus, for any split
augmented simplicial object $A_{\bullet} \to A_{-1}$ in
$\Ss_{ac}(\C)$, and any simplicial object $X_{\bullet}$ in
$\Ss(\C)$, we have a natural bijection
\[\theta_{A_{\bullet},X_{\bullet}} \maps Hom_{\Ss(\C)}(U(A_{\bullet}),X_{\bullet})
\overset{\simeq}{\to}
Hom_{\Ss_{as}(\C)}(A_{\bullet},Dec(X_{\bullet}))\] which takes any
simplicial map $f_{\bullet} \maps U(A_{\bullet}) \to X_{\bullet}$ to
its composite with the splitting
\[\xymatrix@!=2pc{A_{-1} \ar@{-->}[drr]_{s_{0}} \ar@/^2pc/[rr]^{s_{0}} &&
A_{0} \ar[ll]_{d} \ar[rr] \ar[d]^{f_{0}}
\ar@{-->}[drr]^(0.6){f_{1}s_{1}} \ar@/^2pc/[rr]^{s_{1}} && A_{1}
\ar@<-0.5ex>[ll]_{d_{1}} \ar@<0.5ex>[ll]^{d_{0}} \ar@<-0.5ex>[rr]
\ar@<0.5ex>[rr] \ar[d]^{f_{1}} \ar@{-->}[drr]^(0.6){f_{2}s_{2}}
\ar@/^2pc/[rr]^{s_{2}} && A_{2} \ar[ll] \ar@<-1ex>[ll]_{d_{2}}
\ar@<1ex>[ll]^{d_{0}} \ar[d]^{f_{2}} \ar[rr] \ar@<-1ex>[rr]
\ar@<1ex>[rr] \ar@{-->}[drr]^(0.6){f_{3}s_{3}}
\ar@/^2pc/[rr]^{s_{3}} && A_{3} \ar[d]^{f_{3}}
\ar@<-0.5ex>[ll] \ar@<0.5ex>[ll] \ar@<-1.5ex>[ll]_{d_{3}} \ar@<1.5ex>[ll]^{d_{0}} \\
&& X_{0} \ar[rr] && X_{1} \ar@<-0.5ex>[ll]_{d_{1}}
\ar@<0.5ex>[ll]^{d_{0}} \ar@<-0.5ex>[rr] \ar@<0.5ex>[rr] && X_{2}
\ar[ll] \ar@<-1ex>[ll]_{d_{2}} \ar@<1ex>[ll]^{d_{0}} \ar[rr]
\ar@<-1ex>[rr] \ar@<1ex>[rr] && X_{3} \ar@<-0.5ex>[ll]
\ar@<0.5ex>[ll] \ar@<-1.5ex>[ll]_{d_{3}} \ar@<1.5ex>[ll]^{d_{0}} }\]
as in the above diagram.
\newpage
In order to compare later our 2-torsors with Glenn's simplicial
2-torsors we will recall some basic definitions from \cite{Gl}.

\begin{defi} A simplicial map $\Lambda_{\bullet} \maps
\E_{\bullet} \to \B_{\bullet}$ is said to be an exact fibration in
dimension $n$, if for all $0 \leq k \leq n$, the diagrams
\[\xymatrix@!=4pc{E_{n} \ar[d]_{p_{\bar{k}}} \ar[r]^-{\lambda_{n}} & B_{n} \ar[d]^{p_{\bar{k}}} \\
\bigwedge^{k}_{n}(\E_{\bullet}) \ar[r]^{} &
\bigwedge^{k}_{n}(\B_{\bullet})}\] are pullbacks. It is called an
exact fibration if it is an exact fibration in all dimensions $n$.
\end{defi}

Using the language of simplicial algebra, Glenn defined actions and
n-torsors over n-dimensional hypergroupoids. This objects morally
play the role of the n-nerve of weak n-groupoids, and we give their
formal definition.

\begin{defi} An n-dimensional Kan hypergroupoid is a Kan simplicial object
$G_{\bullet}$ in $\E$ such that the canonical map $G_{m} \to
\bigwedge^{k}_{m}(G_{\bullet})$ is an isomorphism for all $m > n$
and $0 \leq k \leq m$.
\end{defi}

\begin{rem} The term n-dimensional hypergroupoid was introduced by Duskin
\cite{Du2}, for any simplicial object satisfying the above condition
without being Kan simplicial object. One of his motivational
examples was the standard simplicial model for an Eilenberg-MacLane
space $K(A,n)$, for any abelian group object $A$ in $\E$. In
\cite{Be2}, Beke used the term an exact n-type to emphasize the
meaning of these objects as algebraic models for homotopy n-types.
\end{rem}

\begin{defi} An action of the n-dimensional hypergroupoid is an internal simplicial
map $\Lambda_{\bullet} \maps \Pe_{\bullet} \to \B_{\bullet}$ in $\E$
which is an exact fibration for all $m \geq n$.
\end{defi}

\begin{defi} An action $\Lambda_{\bullet} \maps P_{\bullet}
\to \B_{\bullet}$ is the n-dimensional hypergroupoid n-torsor over
$X$ in $\E$ if $P_{\bullet}$ is augmented over $X$, aspherical and
n-1-coskeletal ($P_{\bullet} \simeq Cosk^{n-1}(P_{\bullet}))$.
\end{defi}

\section{Bicategories}
Bicategories were defined by Benabou \cite{Be}, and from the modern
perspective, we could call them weak 2-categories. Instead of
stating their original definition we will use Batanin's approach to
weak n-categories given in \cite{Ba}. In this approach a bicategory
$\B$, given by the reflexive 2-graph

\[ \xymatrix{\B \equiv (B_{2} \ar@<1ex>[r]^-{d^{1}_{1}}
\ar@<-1ex>[r]_-{d^{1}_{0}} & B_{1} \ar[l] \ar@<1ex>[r]^-{d^{0}_{1}}
\ar@<-1ex>[r]_-{d^{0}_{0}} & B_{0} \ar[l])}\] \\is a 1-skeletal
monoidal globular category, given by the diagram of categories and
functors
\[\xymatrix{\B_{1} \ar@<1ex>[r]^{D_{1}} \ar@<-1ex>[r]_{D_{0}} & \B_{0} \ar[l]}\]
where the category $\B_{1}$ is the category of morphisms of the
bicategory $\B$ and the category $\B_{0}$ is the image $\D(B_{0})$
of the discrete functor $\D \maps \Set \to \Cat$ which just turns an
object of $\E$ into a discrete internal category in $\E$. Source
functor $D_{1}$ is defined by $D_{1}:=d^{0}_{1} \maps B_{1} \to
B_{0}$ and $D_{1}:=d^{0}_{1}d^{1}_{1}=d^{0}_{1}d^{1}_{0} \maps B_{2}
\to B_{0},$ and a target functor $D_{0}$ is defined by
$D_{0}:=d^{0}_{1} \maps B_{1} \to B_{0}$ and
$D_{1}:=d^{0}_{0}d^{1}_{1}=d^{0}_{0}d^{1}_{1} \maps B_{2} \to
B_{0}$, where we used the same notation for objects and morphisms
parts of the functor. Also, the unit functor $I \maps B_{0} \to
B_{1}$ is defined by $I:=s_{0} \maps B_{0} \to B_{1}$ on the level
of objects, and $I:=s_{1} \maps B_{1} \to B_{2}$ on the level of
morphisms, where $s_{0} \maps B_{0} \to B_{1}$ and $s_{1} \maps
B_{1} \to B_{2}$ are section morphisms in the above 2-graph from
left to right, which we didn't label to avoid too much indices.

In the lower definition of a bicategory we will denote the vertex
$\B_{1} \times_{\B_{0}} \B_{1}$ of the following pullback of
functors
\[\xymatrix@!=6pc{\B_{1} \times_{\B_{0}} \B_{1} \ar[r]^{Pr_{2}} \ar[d]_{Pr_{1}} & \B_{1} \ar[d]^{D_{0}}\\
\B_{1} \ar[r]_{D_{1}} & \B_{0}}\]  by $\B_{2}:=\B_{1}
\times_{\B_{0}} \B_{1}$ and likewise $\B_{3}:=\B_{1} \times_{\B_{0}}
\B_{1} \times_{\B_{0}} \B_{1}$, and so on. Thus we will adopt the
following convention: for any functor $P \maps \E \to B_{0}$, the
first of the symbols
\[\xymatrix{\E \times_{\B_{0}} \B_{1}} \xymatrix{and} \xymatrix{\B_{1} \times_{\B_{0}}
\E}\] will denote the pullback of $P$ and $D_{0}$, and the second
one that of $D_{1}$ and $P$.
\begin{defi}
A {\it bicategory} $\B$ consists of the following data:
\begin{itemize}

\item two categories, a discrete category $\B_{0}$ of objects, and
a category $\B_{1}$ of morphisms of the weak 2-category $\B$,

\item functors $D_{0},D_{1} \maps \B_{1} \to \B_{0}$, called
target and source functors, respectively, a functor $I \maps \B_{0}
\to \B_{1}$, called unit functor, and a functor $H \maps \B_{2} \to
\B_{1}$, called the horizontal composition functor,

\item natural isomorphism
\[\xymatrix@!=6pc{\B_{3} \ar[r]^{H \times Id_{\B_{1}}}
\ar[d]_{Id_{\B_{1}} \times H} \drtwocell<\omit>{<0>\alpha} & \B_{2} \ar[d]^{H}\\
\B_{2} \ar[r]_{H} & \B_{1}}\]

\item natural isomorphisms
\[\xymatrix@!=6pc{& \B_{2} \ar[d]_{H} \drtwocell<\omit>{<5>\rho} &\\
\B_{1} \urtwocell<\omit>{<5>\lambda}\ar@{=}[r] \ar[ur]^{S_{1}} &
\B_{1} \ar@{=}[r] & \B_{1} \ar[ul]_{S_{0}}}\] where the functor
$S_{0} \maps \B_{1} \to \B_{2}$ is defined by the composition
\[\xymatrix{\B_{1} \ar[rr]^-{(D_{0},Id_{\B_{1}})} &&
\B_{1} \times_{\B_{0}} \B_{0} \ar[rr]^{I \times Id_{\B_{1}}} &&
\B_{1} \times_{\B_{0}} \B_{1},}\] and the functor $S_{1} \maps
\B_{1} \to \B_{2}$ is defined by the composition
\[\xymatrix{\B_{1} \ar[rr]^-{(Id_{\B_{1}},D_{1})} &&
\B_{0} \times_{\B_{0}} \B_{1} \ar[rr]^{Id_{\B_{1}} \times I} &&
\B_{1} \times_{\B_{0}} \B_{1},}\] or more explicitly for any
1-morphism $f \maps x \to y$ in $\B$ (i.e. object in $\B_{1}$) we
have $S_{0}(f)= (f,i_{x})$ and $S_{1}(f)= (i_{y},f)$,
\end{itemize}
such that following axioms are satisfied:
\begin{itemize}
\item associativity 3-cocycle
\[\xymatrix@!=5pc{\B_{4}
\ddrrtwocell<\omit>{<15> Id_{\B_{1}} \times \alpha}
\ddrrtwocell<\omit>{<-15> \alpha \times Id_{\B_{1}}} \ar[rr]^{H
\times Id_{\B_{2}}} \ar[dr]^{Id_{\B_{1}} \times H \times
Id_{\B_{1}}} \ar[dd]_{Id_{\B_{2}} \times H} &&
\B_{3} \ar[dr]^{H \times Id_{\B_{1}}} \ar'[d][dd]^>(0.4){Id_{\B_{1}} \times H } &\\
& \B_{3} \drtwocell<\omit>{<-2> \simeq} \ar[rr]_>(0.7){H \times
Id_{\B_{1}}} \ar[dd]^>(0.7){Id_{\B_{1}} \times H}
& \drtwocell<\omit>{<5> \alpha} & \B_{2} \ar[dd]^{H} \\
\B_{3} \ar[dr]_{Id_{\B_{1}} \times H} \ar'[r][rr]^>(0.4){H \times
Id_{\B_{1}}} && \B_{2} \ar[dr]^{H}
\drtwocell<\omit>{<-7> \alpha} \drtwocell<\omit>{<7> \alpha}  &\\
& \B_{2} \ar[rr]_{H} && \B_{1}}\]

which for any object $(k,h,g,f)$ in $\B_{4}$ becomes the commutative
pentagon
\[\xymatrix@!=3pc{&&((k \circ h) \circ g) \circ f \ar[dll]_-{\alpha_{k,h,g} \circ f}
\ar[drr]^-{\alpha_{k \circ h,g,f}} &&\\
(k \circ (h \circ g)) \circ f \ar[ddr]_-{\alpha_{k,h \circ g,f}}&&&&
(k \circ h) \circ (g \circ f) \ar[ddl]^-{\alpha_{k,h,g \circ f}}\\&&&&\\
& k \circ ((h \circ g) \circ f)) \ar[rr]_-{k \circ \alpha_{h,g,f}}&&
k \circ (h \circ (g \circ f))&}\] of components of natural
transformations

\item the commutative pyramid
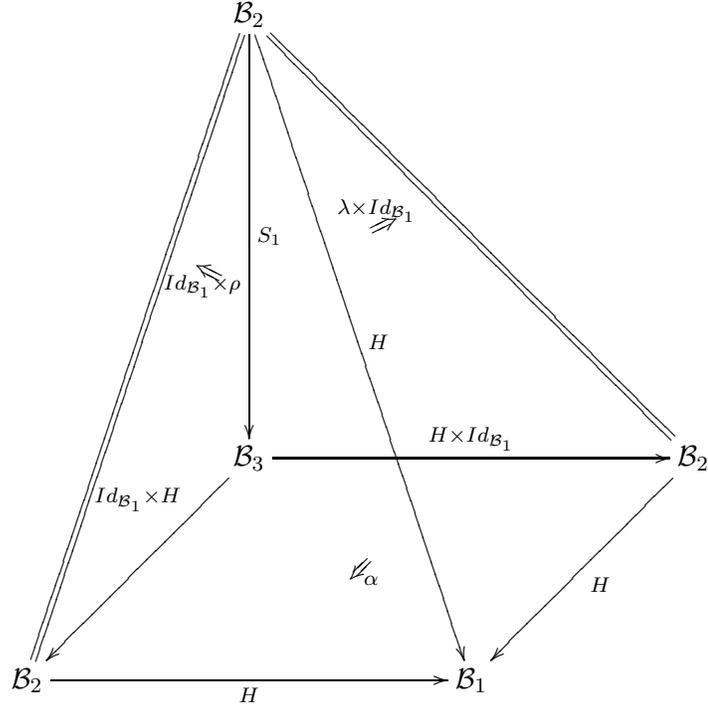
\[\xymatrix@!=5pc{& \B_{2} \ddltwocell<\omit>{<-6> Id_{\B_{1}} \times \rho}
\ar[dd]^{S_{1}} \ar@{=}[dddl] \ar[dddr]^{H} \ar@{=}[ddrr] &&\\
&&&\\
& \B_{3} \drtwocell<\omit>{<0> \alpha} \ar[dl]_>(0.3){Id_{\B_{1}}
\times H} \ar[rr]^{H \times Id_{\B_{1}}} & \uultwocell<\omit>{<2>
\lambda \times Id_{\B_{1}}}& \B_{2} \ar[dl]^{H}\\
\B_{2} \ar[rr]_{H} && \B_{1} & }\] which for any object $(g,f)$ in
$\B_{2}$ becomes the triangle diagram
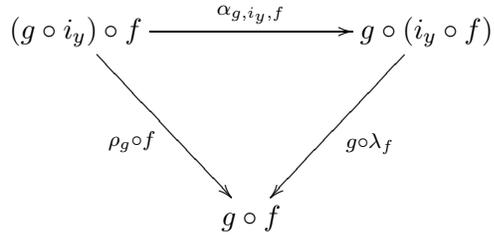
\[\xymatrix{(g \circ i_{y}) \circ f \ar[ddr]_{\rho_{g} \circ f} \ar[rr]^-{\alpha_{g,i_{y},f}} &&
g \circ (i_{y} \circ f) \ar[ddl]^-{g \circ \lambda_{f}}\\
&&\\
& g \circ f &}\]
\end{itemize}
\end{defi}

\begin{rem} Note that in the above definition of the horizontal
composition functor $H \maps \B_{2} \to \B_{1}$, for any diagram of
2-arrows (i.e. a morphism in a category $\B_{2} \times_{\B_{1}}
\B_{2}$)
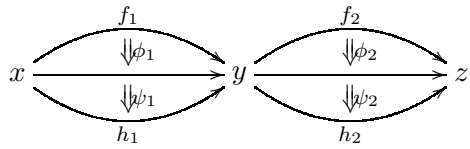
\[\xymatrix@!=5pc{x \ruppertwocell^{f_{1}}{\phi_{1}}
\rlowertwocell_{h_{1}}{\psi_{1}} \ar[r] & y
\ruppertwocell^{f_{2}}{\phi_{2}} \rlowertwocell_{h_{2}}{\psi_{2}}
\ar[r] & z}\] by functoriality we immediately have a Godement
interchange law
\[\xymatrix{(\psi_{2} \circ \psi_{1})(\phi_{2} \circ \phi_{1}) =
(\psi_{2} \psi_{1})  \circ (\phi_{2} \phi_{1}). }\]
\end{rem}

\begin{ex}(Strict 2-categories) A weak 2-category in which
associativity and left and right identity natural isomorphisms are
identities is called (strict) 2-category.
\end{ex}

\begin{ex}(Monoidal categories) Monoidal category is a bicategory $\B$
in which $\B_{0}=1$ is terminal discrete category (or one point
set). Strict monoidal category is a one object strict 2-category.
\end{ex}

\begin{ex}(Bicategory of spans) Let $\C$ be a cartesian
category (that is a category with pullbacks). First we make a choice
of the pullback
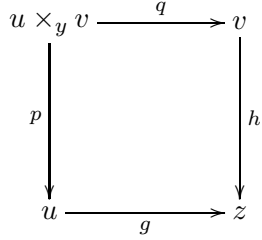
\[\xymatrix@!=4pc{u \times_{y} v \ar[d]_{p} \ar[r]^-{q} & v \ar[d]^{h}\\
u \ar[r]_-{g} & z}\] for any such diagram $x
\overset{f}{\rightarrow} z \overset{g}{\leftarrow} y$ in a category
$\C$. We construct the weak 2-category $Span(\C)$ of spans in the
category $\C$. The objects of $Span(\C)$ are the same as objects of
$\C$. For any two objects $x,y$ in $Span(\C)$, a 1-morphism $u \maps
x \nrightarrow y$ is a span
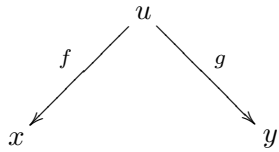
\[\xymatrix@!=2pc{& u \ar[dl]_{f} \ar[dr]^{g} &\\ x && y}\] and a
2-morphism $a \maps z \nRightarrow w$ is given by the commutative
diagram
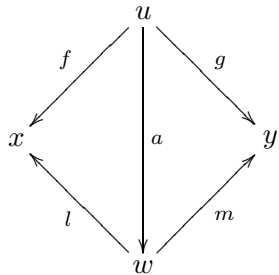
\[\xymatrix@!=2pc{& u \ar[dl]_{f} \ar[dr]^{g} \ar[dd]^{a} &\\ x && y \\
& w \ar[ul]^{l} \ar[ur]_{m} & }\] from which we easily see that
vertical composition of 2-morphisms is given by the composition in
$\C$. Horizontal composition of composable 1-morphisms
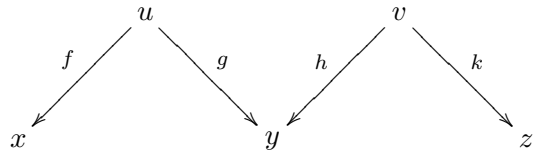
\[\xymatrix@!=2pc{& u \ar[dl]_{f} \ar[dr]^{g} && v \ar[dl]_{h} \ar[dr]^{k} &\\ x && y &&
z}\] is given by the pullback
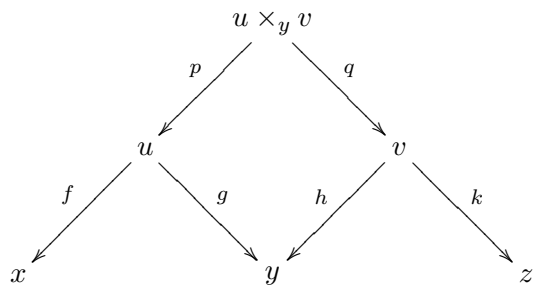
\[\xymatrix@!=2pc{&& u \times_{y} v \ar[dl]_{p} \ar[dr]^{q} &&\\
& u \ar[dl]_{f} \ar[dr]^{g} && v \ar[dl]_{h} \ar[dr]^{k} &\\ x && y
&& z}\] and from here we have obvious horizontal identity $i_{x}
\maps x \nrightarrow x$
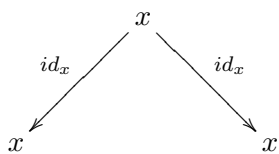
\[\xymatrix@!=2pc{& x \ar[dl]_{id_{x}} \ar[dr]^{id_{x}} &\\ x && x}\]
\end{ex}

\begin{ex}(Bimodules) Let $Bim$ denote the bicategory whose objects
are rings with identity. For any two rings $A$ and $B$, $Bim(A,B)$
will be a category of $A-B$ bimodules and their homomorphisms.
Horizontal composition is given by the tensor product, and
associativity and identity constraints are the usual ones for the
tensor product.
\end{ex}

\section{Nerves of bicategories}

In this section, we describe the nerve construction for
bicategories, first given by Duskin in \cite{Du5}. This construction
is a natural outcome of various attempts to describe nerves of
higher dimensional categories and groupoids, whose origin is a
conjecture on a characterization of the nerve of {\it strict}
n-category, in an unpublished work of Roberts. This conjecture was
published by Street in \cite{St4}, and it was finally proved by
Verity \cite{V1}, who characterized nerves of strict n-categories by
means of special simplicial sets, which he called {\it complicial
sets}.

We will derive the construction of the Duskin nerve for bicategories
from the standard description of the geometric nerve. First we have
a fully faithful functor
\begin{equation} \label{embbicat}
\xymatrix{i \maps \Delta \to Bicat}
\end{equation}
where $Bicat$ is a category of bicategories and their homomorphisms,
as it is given in \cite{Be}, so we consider each ordinal as a
locally discrete 2-category. Thus the nerve of the bicategory $\B$
is a simplicial set $N_{2}\B_{\bullet}$ which is defined via the
embedding (\ref{embbicat}) by
\begin{equation} \label{Dusnerve}
\xymatrix{N_{2}\B_{n}:=Hom_{Bicat}(i[n],\B).}
\end{equation}
The 0-simplices of $N_{2}(\B)$ are the objects of $\B$ and
1-simplices are directed line segments
\[\xymatrix{x_{0} \ar[r]^-{f_{01}} & x_{1}}\] which may be seen as
homomorphisms $f \maps [1] \to \B$ from the locally discrete
bicategory $[1]$ to $\B$. Face maps are defined by
$d_{0}(f_{01})=x_{1}$ and $d_{1}(f_{01})=x_{0}$. If $x_{0}$ is a
0-cell of $\B$ then we define the corresponding degenerate 1-simplex
$s_{0}(x_{0})$ by
\[\xymatrix{x_{0} \ar[r]^-{id_{x_{0}}} & x_{0}.}\] A typical 2-simplex is given by the
triangle filled with a 2-morphism $\beta_{012} \maps f_{12} \circ
f_{01} \Rightarrow f_{02}$
\[\xymatrix@!=4pc{x_{0} \drtwocell<\omit>{<-4>\,\,\,\,\,\,\,\beta_{012}}
\ar[r]^-{f_{1}} \ar[dr]_-{f_{12}} & x_{1} \ar[d]^-{f_{2}}\\
& x_{2}}\] where $f_{ij} \maps [1] \to \B$ is a homomorphism for
which $f_{ij}(0)=x_{i}$ and $f_{ij}(1)=x_{j}$. The face operators
are defined as usual by
\[d_{i}(f_{12},f_{02},f_{01},\beta_{012})= \left\lbrace
\begin{array}{l}
f_{12} \hspace{2cm} i=0\\
f_{02} \hspace{2cm} i=1\\
f_{01} \hspace{2cm} i=2
\end{array}\right. \] while for a 1-cell $\xymatrix{x_{0} \ar[r]^-{f_{01}} &
x_{1}}$ the degeneracy operators are defined by
\[\begin{array}{c}
s_{0}(f_{01})=\rho_{f_{01}}\\
s_{1}(f_{01})=\lambda_{f_{01}}
\end{array}\] which are the two 2-simplices
\[\xymatrix@!=4pc{x_{0} \drtwocell<\omit>{<-4>\,\,\,\,\,\,\,\rho_{f_{01}}}
\ar[r]^-{id_{x_{0}}} \ar[dr]_-{f_{01}} & x_{0} \ar[d]^-{f_{01}}\\
& x_{1}} \hspace{3pc} \xymatrix@!=4pc{x_{0}
\drtwocell<\omit>{<-4>\,\,\,\,\,\,\,\lambda_{f_{01}}}
\ar[r]^-{f_{01}} \ar[dr]_-{f_{01}} & x_{1} \ar[d]^-{id_{x_{1}}}\\
& x_{1}}\] respectively, where the 1-morphisms $\rho_{f_{01}} \maps
f_{01} \circ id_{x_{0}} \to f_{01}$ and $\lambda_{f_{01}} \maps
id_{x_{1}} \circ f_{01} \to f_{01}$ are the components of the right
and left identity natural isomorphisms in $\B$. The general
3-simplex is of the form
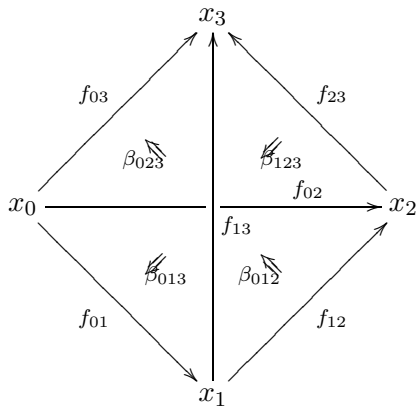
\[\xymatrix@!=4pc{& x_{3} \drtwocell<\omit>{<4>\beta_{123}}  \dltwocell<\omit>{<-4>\beta_{023}} &\\
x_{0} \drtwocell<\omit>{<-4>\beta_{013}} \ar[ur]^{f_{03}}
\ar'[r][rr]^{f_{02}} \ar[dr]_-{f_{01}}  &&
x_{2} \dltwocell<\omit>{<4>\beta_{012}} \ar[ul]_-{f_{23}} \\
& x_{1} \ar[uu]_(0.45){f_{13}} \ar[ur]_-{f_{12}} &}\] such that we
have an identity
\[\beta_{023}(\beta_{012} \circ f_{23})\alpha_{0123}= \beta_{013}(\beta_{123} \circ f_{01})\]
where $\alpha_{0123} \maps (f_{23} \circ f_{12}) \circ f_{01}
\Rightarrow f_{23} \circ (f_{12} \circ f_{01})$, and this condition
follows directly from the coherence for the composition. Since this
construction is given by the geometric nerve (\ref{Dusnerve}) it
follows immediately that the Duskin nerve is functorial with respect
to homomorphisms of bicategories, which leads us to the following
result.

\begin{thm} The Duskin nerve functor $N_{2} \maps Bicat \to \Ss Set$ is fully
faithful.
\begin{proof} An analogous proof that the geometric nerve provides a
fully faithful functor on the category $2-Cat_{lax}$ of 2-categories
and normal lax 2-functors is given in \cite{BBF1}. Then the
statement of the theorem follows immediately for a category $Bicat$
of bicategories and  normal homomorphisms.
\end{proof}
\end{thm}

\section{Actions of bicategories}

Now, we will introduce actions of bicategories. It will be clear
from the definition that such actions are categorification of
actions of categories.

\begin{defi}
A right action of a bicategory $\B$ is quintuple
$(\C,\Lambda,A,\kappa,\iota)$ given by:

\begin{itemize}
\item a category $\C$ and a functor $\Lambda \maps \C \to \B_{0}$
to the discrete category of objects $\B_{0}$ of the bicategory $\B$,
called the momentum functor,

\item a functor $A \maps \C \times_{\B_{0}} \B_{1} \to \C$,
called the action functor, and we usually write $A(p,f):=p \tril f$,
for any object $(p,f)$ in $\C \times_{\B_{0}} \B_{1}$, and
$A(a,\phi):=a \tril \phi$ for any morphism $(a,\phi) \maps (p,f) \to
(q,g)$ in $\C \times_{\B_{0}} \B_{1}$,

\item a natural isomorphism
\[\xymatrix@!=8pc{\C \times_{\B_{0}} \B_{1} \times_{\B_{0}} \B_{1} \ar[d]_-{Id_{\C} \times
D_{1}} \ar[r]^-{A \times Id_{\B_{1}}} \drtwocell<\omit>{\kappa}
& \C \times_{\B_{0}} \B_{1} \ar[d]^{A}\\
\C \times_{\B_{0}} \B_{1}  \ar[r]_{A} & \C}\] whose components are
denoted by $\kappa_{p,f,g} \maps (p \tril f)\tril g \to p \tril (f
\circ g)$ for any object $(p,f,g)$ in $\C \times_{\B_{0}} \B_{1}
\times_{\B_{0}} \B_{1}$

\item a natural isomorphism
\[\xymatrix@!=4pc{& \C \times_{\B_{0}} \B_{1} \ar[dr]^{A} &\\
\C \rrtwocell<\omit>{<-4>\iota} \ar[ur]^-{(Id_{\C},I\Lambda)}
\ar@{=}[rr] && \C}\] whose components are denoted by $\iota_{p}
\maps p \tril i_{\Lambda(p)} \to p$ for each object $p$ in $\C$
\end{itemize}
such that following axioms are satisfied:
\begin{itemize}
\item equivariance of the action
\[\xymatrix@!=5pc{\C \times_{\B_{0}} \B_{1} \ar[r]^-{A} \ar[d]_-{Pr_{2}} &
\C \ar[d]^-{\Lambda}\\
\B_{1} \ar[r]_{D_{1}} & \B_{0}}\]

which means that for any object $(p,f)$ in $\C \times_{\B_{0}}
\B_{1}$, we have $\Lambda(p \tril f)=D_{1}(f)$, and for any morphism
$(a,\phi) \maps (p,f) \to (q,g)$ in $\C \times_{\B_{0}} \B_{1}$, we
have $\Lambda(a \tril \phi)=D_{1}(\phi)$,

\item for any object $(p,f,g,h)$ in $\C \times_{\B_{0}} \B_{1}
\times_{\B_{0}} \B_{1} \times_{\B_{0}} \B_{1}$ the following diagram

\[\xymatrix@!=2pc{&&((p \tril f) \tril g) \tril h
\ar[dll]_-{\kappa_{p,f,g} \tril h} \ar[drr]^-{\kappa_{p \tril f,g,h}} &&\\
(p \tril (f \circ g)) \tril h \ar[ddr]_-{\kappa_{p,f \circ g,h}}
&&&&
(p \tril f) \tril (g \circ h) \ar[ddl]^-{\kappa_{p,f,g \circ h}}\\&&&&\\
& p \tril ((f \circ g) \circ h)) \ar[rr]_-{p \tril \alpha_{f,g,h}}
&& p \tril (f \circ (g \circ h)) &}\] commutes,

\item for any object $(p,f)$ in $\C \times_{\B_{0}} \B_{1}$
following diagrams

\[\xymatrix@!=2pc{(p \tril i_{\Lambda_{0}(p)}) \tril f \ar[ddr]_{\iota_{p} \tril f}
\ar[rr]^-{\kappa_{p,i_{\Lambda_{0}(p)},f}} &&
p \tril (i_{\Lambda_{0}(p)} \circ f) \ar[ddl]^-{p \tril \lambda_{f}}\\
&&\\
& p \tril f &} \hspace{1pc} \xymatrix@!=2pc{(p \tril f) \tril
i_{s_{0}(f)} \ar[ddr]_{\iota_{p \tril f} \tril i_{s_{0}(f)}}
\ar[rr]^-{\kappa_{p,f,i_{s_{0}(f)}}} &&
p \tril (f \circ i_{s_{0}(f)}) \ar[ddl]^-{p \tril \rho_{f}}\\
&&\\
& p \tril f &}\] commute.
\end{itemize}
\end{defi}

\begin{rem} Note the fact that $A \maps \C \times_{\B_{0}} \B_{1} \to
\C$ is a functor, immediately implies an interchange law
\[\xymatrix{(b \tril \psi)(a \tril \phi)=(ba) \tril (\psi \phi)}\]
\end{rem}

\begin{defi} Let $\pi \maps \C \to M$ be a bundle of categories over an object $M$ in $\E$.
A (fiberwise) right action of a bicategory $\B$ on a bundle of
categories $\pi \maps \C \to M$ is given by the action of the
bicategory $\B$ on a category $\C$ for which the diagram
\[\xymatrix@!=6pc{\C \times_{\B_{0}} \B_{1} \ar[r]^-{A} \ar[d]_-{Pr_{1}} &
\C \ar[d]^-{\pi}\\
\B_{1} \ar[r]_{\pi} & M}\] commute. We call a bundle $\pi \maps \C
\to M$, a $\B$-2-bundle over $M$.
\end{defi}

\begin{defi} Let $(\C,\Lambda,A,\kappa,\iota)$ and $(\D,A',\Omega,\kappa',\iota')$
be two $\B$-categories. A $\B$-equivariant functor is a pair
$(F,\theta) \maps (\C,\Lambda,A,\kappa,\iota) \to
(\D,A',\Omega,\kappa',\iota')$ consisting of
\begin{itemize}
\item a functor $F \maps \C \to \D$

\item a natural transformations $\theta \maps A' \circ (F \times Id_{\B_{1}}) \Rightarrow
F \circ A$
\[\xymatrix@!=6pc{\C \times_{\B_{0}} \B_{1} \drtwocell<\omit>{\theta}
\ar[r]^-{F \times Id_{\B_{1}}} \ar[d]_-{A}
& \D \times_{\B_{0}} \B_{1} \ar[d]^{A'}\\
\C  \ar[r]_{F} & \D}\]
\end{itemize}
such that following conditions are satisfied:
\begin{itemize}
\item $\Omega \circ F= \Lambda$
\[\xymatrix@!=3pc{\C \ar[dr]_-{\Lambda} \ar[rr]^-{F}
&& \D \ar[dl]^{\Omega}\\
& \B_{0} &}\]

\item the diagram
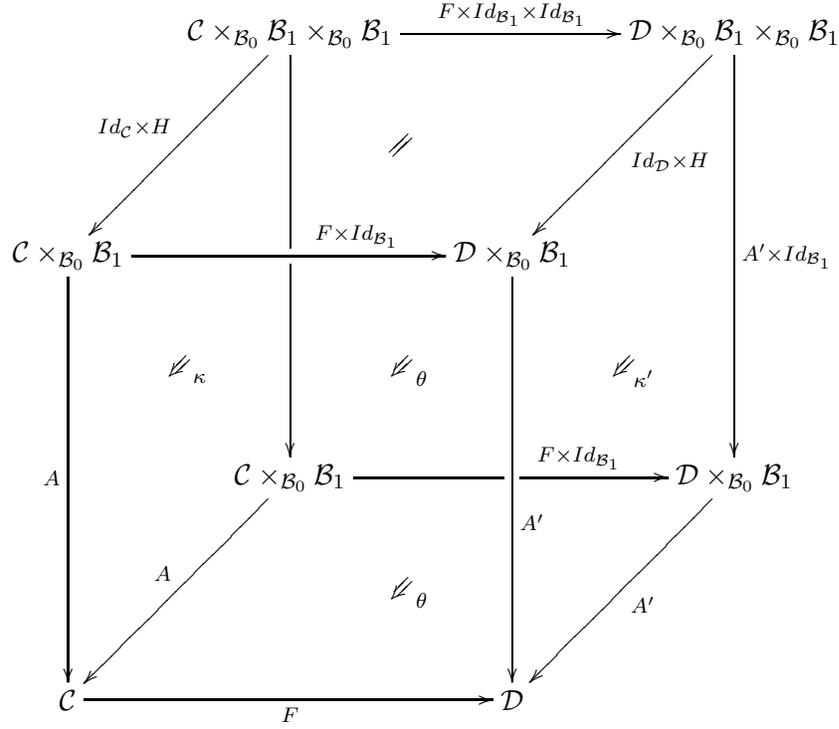
\[\xymatrix@!=5pc{& \C \times_{\B_{0}} \B_{1} \times_{\B_{0}} \B_{1}
\drtwocell<\omit>{ =} \ar[dl]_{Id_{\C} \times H} \ar[rr]^{F \times
Id_{\B_{1}} \times Id_{\B_{1}}} \ar'[d][dd] &&
\D \times_{\B_{0}} \B_{1} \times_{\B_{0}} \B_{1} \ar[dl]^{Id_{\D} \times H} \ar[dd]^{A' \times Id_{\B_{1}}} \\
\C \times_{\B_{0}} \B_{1} \ar[rr]^(0.65){F \times Id_{\B_{1}}}
\ar[dd]_{A}\drtwocell<\omit>{\,\,\,\,\,\kappa} &
\drtwocell<\omit>{\,\,\,\,\,\theta}
& \D \times_{\B_{0}} \B_{1} \ar[dd]^(0.6){A'} \drtwocell<\omit>{\,\,\,\,\,\kappa'} & \\
& \C \times_{\B_{0}} \B_{1} \drtwocell<\omit>{\,\,\,\,\,\theta}
\ar[dl]_{A} \ar'[r][rr]^(0.3){F \times Id_{\B_{1}}} &&
\D \times_{\B_{0}} \B_{1} \ar[dl]^{A'} \\
\C \ar[rr]_{F}  && \D } \] commutes, which means that we have an
identity of natural transformations
\[\hspace{-1cm} (F \circ \kappa)
[\theta \circ (A \times Id_{\B_{1}})] [A' \circ (\theta \times
Id_{\B_{1}})] = [\theta \circ (Id_{\C} \times H)] [\kappa' \circ (F
\times Id_{\B_{1}} \times Id_{\B_{1}})]\] when evaluated at object
$(p,f,g)$ in $\C \times_{\B_{0}} \B_{1} \times_{\B_{0}} \B_{1}$,
becomes a commutative diagram
\[\xymatrix@!=6pc{(F(p) \tril f) \tril g \ar[r]^{\theta_{p,f} \tril g} \ar[d]_{\kappa'_{F(p),f,g}} &
F(p \tril f) \tril g \ar[r]^{\theta_{p \tril f,g}} & F((p \tril f)
\tril g) \ar[d]^{F(\kappa_{p,f,g})}\\
F(p) \tril (f \circ g) \ar[rr]_{\theta_{p,f \circ g}} && F(p \tril
(f \circ g))}\] in the category $\D$.
\item the diagram
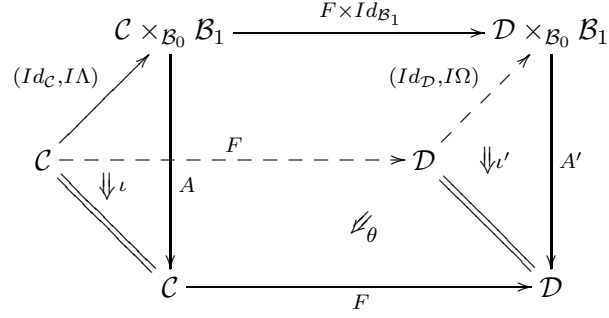
\[\xymatrix@!=2pc{& \C \times_{\B_{0}} \B_{1} \ar[dd]^(0.6){A}
\ar[rrr]^{F \times Id_{\B_{1}}}
&&& \D \times_{\B_{0}} \B_{1} \ar[dd]^{A'}\\
\C \rtwocell<\omit>{<2>\iota} \ar[ur]^{(Id_{\C},I\Lambda)}
\ar@{-->}[rrr]^{F} \ar@{=}[dr] && \drtwocell<\omit>{\theta} &
\D \rtwocell<\omit>{\iota'} \ar@{-->}[ur]^{(Id_{\D},I\Omega)} \ar@{=}[dr] &\\
& \C \ar[rrr]_{F} &&& \D }\] commutes, which means that we have
identity of natural transformations
\[(\iota' \circ Id_{F})Id_{F} = (F \circ \iota) [\theta \circ (Id_{\C},I\Lambda)] Id_{(F,I\Lambda)}\]
when evaluated at object $p$ in $\C$, becomes a commutative diagram
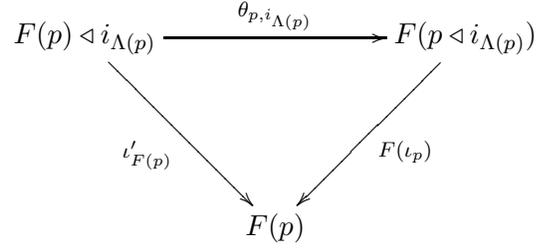
\[\xymatrix@!=4pc{F(p) \tril i_{\Lambda(p)} \ar[rr]^{\theta_{p,i_{\Lambda(p)}}}
\ar[dr]_{\iota'_{F(p)}}
&& F(p \tril i_{\Lambda(p)}) \ar[dl]^{F(\iota_{p})} \\
& F(p)&}\] in the category $\D$.
\end{itemize}
\end{defi}

\begin{defi} A $\B$-equivariant natural transformation $\tau \maps (F,\theta) \Rightarrow (G,\zeta)$
between $\B$-covariant functors $(F,\theta),(G,\zeta) \maps
(\C,\Lambda,\Phi,\alpha,\iota) \to (\D,\Psi,\Omega,\beta,\kappa)$ is
a natural transformation $\tau \maps F \Rightarrow G$ such that
diagram
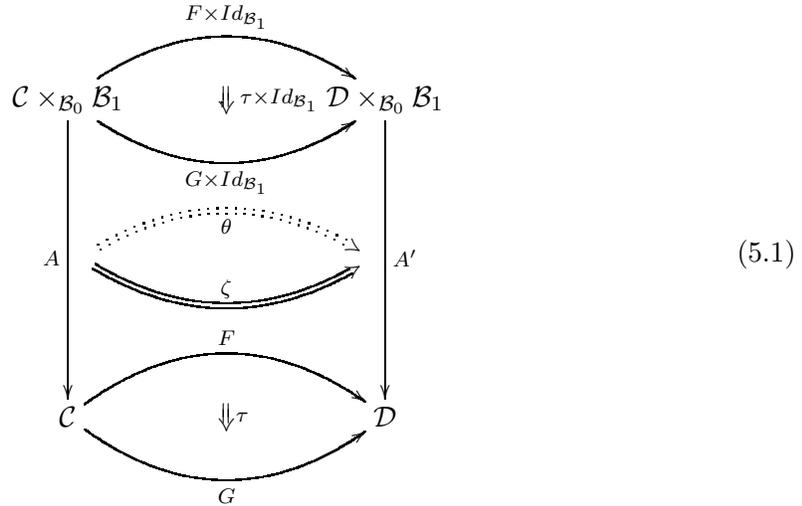
\begin{equation}
\begin{array}{l}\label{eqnat}
\xymatrix@!=3pc{\C \times_{\B_{0}} \B_{1} \rrtwocell<\omit>{<0>
\,\,\,\,\,\,\,\,\,\,\,\,\,\,,\,\,\,\,\tau \times Id_{\B_{1}}}
\ar@/_2pc/[rr]_{G \times Id_{\B_{1}}} \ar@/^2pc/[rr]^{F \times
Id_{\B_{1}}} \ar[dd]_{A}^{\ }="s" & & \D \times_{\B_{0}} \B_{1}
\ar[dd]^{A'}_{\ }="t"
\\&&\\
\C \rrtwocell<\omit>{<0> \tau} \ar@{->}@/^2pc/[rr]^{F}
\ar@/_2pc/[rr]_{G} && \D \ar@{=>}@/_1.5pc/"s";"t"^{\zeta} |>\hole
\ar@{:>}@/^1.5pc/"s";"t"_{\theta} |>\hole}
\end{array}
\end{equation}
commutes, which means that we have a following identity
\[ \zeta [A' \circ (\tau \times
Id_{\B_{1}})] = (\tau \circ A) \theta \] that becomes a commutative
diagram
\[\xymatrix@!=4pc{F(p) \tril f \ar[r]^{\theta_{p,f}}
\ar[d]_{\tau_{p} \tril f} & F(p \tril f) \ar[d]^{\tau_{p \tril f}} \\
G(p) \tril f \ar[r]_{\zeta_{p,f}} & G(p \tril f)}\] in the category
$\D$, when evaluated at object $p$ in $\C$.
\end{defi}

The above construction gives rise to the 2-category in an obvious
way, so we have a following theorem.

\begin{thm} The class of $\B$-categories, $\B$ equivariant
functors and their natural transformations form a 2-category.
\end{thm}
\begin{proof} The vertical and horizontal composition in a
2-category is induced from the composition in $\Cat$.
\end{proof}

\noindent Let $\B$ be a bicategory and $\Pe$ a category together
with a momentum functor $\Lambda \maps \Pe \to \B_{0}$
\begin{equation}
\begin{array}{l}\label{biaction}
\xymatrix@!=3pc{P_{1} \ar@<1ex>[d]^-{s} \ar@<-1ex>[d]_{t} & B_{2}
\ar@<1ex>[d]^{s_{1}} \ar@<-1ex>[d]_{t_{1}}\\ P_{0}
\ar[dr]_{\Lambda_{0}} & B_{1}
\ar@<1ex>[d]^{s_{0}} \ar@<-1ex>[d]_{t_{0}}\\
& B_{0}}
\end{array}
\end{equation} and let $\B$ acts on $\Pe$ via an action functor
\begin{equation}
\begin{array}{l}\label{2action}
A \maps \Pe \times_{\B_{0}} \B_{1} \to \Pe
\end{array}
\end{equation}
which satisfies coherence axioms from Definition 13.1. Such actions
allows us to introduce a fundamental objects which we will use
later.

\begin{ex} (Regular Lie groupoids) Recall that a regular Lie groupoid $\G$ is Lie groupoid
\begin{equation}
\begin{array}{l}\label{reglie}
\xymatrix{G \ar@<1ex>[r]^-{s} \ar@<-1ex>[r]_-{t} & M \ar[l]}
\end{array}
\end{equation}
such that for each $x \in M$, the target map $t \maps G \to M$
restricts to a map $t \maps s^{-1}(x) \to M$ of locally constant
rank. Regular Lie groupoids cover many important classes of Lie
groupoids like transitive Lie groupoids, \'etale Lie groupoids and
bundles of Lie groups. Groupoids which arise in foliation theory are
always regular, and Moerdijk gave a classification of regular Lie
groupoids in \cite{Mo}. Any regular Lie groupoid (\ref{reglie}) fits
into extension of Lie groupoids
\begin{equation}
\begin{array}{l}\label{regext}
\xymatrix@!=2pc{K \ar[drr]_{k} \ar[rr]^{j} && G \ar@<.8ex>[d]^-{s}
\ar@<-.8ex>[d]_-{t} \ar[rr]^{\pi} &&
E \ar@<.8ex>[dll]^-{s} \ar@<-.8ex>[dll]_-{t}\\
&& M &&}
\end{array}
\end{equation} where $E$ is the manifold of morphisms of a foliation
groupoid $\E$ over $M$, which means that the fibers of the map
$(s,t) \maps E \to M \times M$ are discrete. Therefore, the category
of regular Lie groupoids $\R Lie(M)$ over $M$ is equipped with a
canonical projection functor
\begin{equation}
\begin{array}{l}\label{projreg}
\xymatrix{\R Lie(M) \ar[d]^-{\Pi} \\ \Fol(M)}
\end{array}
\end{equation}
to the category $\Fol(M)$ of foliation groupoids over $M$. There is
a bigroupoid $\B(M)$ whose discrete groupoid of objects BunGr(M)
consists of bundles of Lie groups over $M$, and whose groupoid of
morphisms $Bitors(M)$ has bitorsors as objects and their
isomorphisms as morphisms. The bigroupoid $\B(M)$ acts on the
projection (\ref{projreg}) with respect to the diagram
\begin{equation}
\begin{array}{l}\label{bitorsact}
\xymatrix@!=4pc{Bitors(M) \ar@<.8ex>[d]^-{S} \ar@<-.8ex>[d]_-{T} & \R Lie(M) \ar[d]^-{\Pi} \ar[dl]^(0.42){\Lambda}\\
BunGr(M) & \Fol(M)}
\end{array}
\end{equation}
in which a momentum functor $\Lambda \maps \R Lie(M) \to BunGr(M)$
is defined for a regular Lie groupoid $G$ by $\Lambda(G)=K$, where
$k \maps K \to M$ is bundle of Lie groups from an extension
(\ref{regext}). The (left) action of the bigroupoid $\B(M)$ on the
category $\R Lie(M)$ is defined by a functor
\begin{equation}
\begin{array}{l}\label{projregact}
\A \maps Bitors(M) \times_{BunGr(M)} \R Lie(M) \to \R Lie(M)
\end{array}
\end{equation}
on pairs $({}_{L}P_{K},G)$ where ${}_{L}P_{K}$ is just the notation
for $L-K$ bitorsor over $M$, and $G$ fits into an exact sequence of
groupoids (\ref{regext}), and produces a new regular Lie groupoid $P
\tensor_{K} G \tensor_{K} P^{-1}$ which fits into an exact sequence
of groupoids
\begin{equation}
\begin{array}{l}\label{regextact}
\xymatrix@!=2pc{L \ar[drr]_{l} \ar[rr]^-{j_{P}} && P \tensor_{K} G
\tensor_{K} P^{-1} \ar@<.8ex>[d]^-{s} \ar@<-.8ex>[d]_-{t}
\ar[rr]^-{\pi_{P}} &&
E \ar@<.8ex>[dll]^-{s} \ar@<-.8ex>[dll]_-{t}\\
&& M &&}
\end{array}
\end{equation}
Here, $P \tensor_{K} G \tensor_{K} P^{-1}$ denotes the usual
contracted product of bitorsors (which also defines a horizontal
composition in the bigroupoid $\B(M)$) whose elements are
equivalence classes $(q \tensor g \tensor p^{-1})$ by the
equivalence relation which identifies points $(qk,g,p^{-1})$ and
$(q,kg,p^{-1})$ (as well as points $(q,gk,p^{-1})$ and
$(q,g,kp^{-1})$) in a fiber product $P \times_{K} G \times_{K}
P^{-1}$. We use the notation ${}_{K}P^{-1}_{L}$ for a $K-L$ bitorsor
opposite to ${}_{L}P_{K}$, which has the same underlying space $P$,
but we denote its points by $p^{-1} \in P^{-1}$ just to distinguish
$P^{-1}$ from $P$. The left action of $K$ on $P^{-1}$ is defined by
means of the right action of $K$ on $P$ by
\[ kp^{-1}=(pk^{-1})^{-1}.\]
Then the map $p_{P} \maps P \times_{K} G \times_{K} P^{-1} \to E$ in
the exact sequence (\ref{regextact}) is defined by
\[ \pi_{P}(q \tensor g \tensor p^{-1})=p(g)\] and its kernel is
an adjoint bundle $P \tensor_{K} P^{-1}$ of Lie groups, whose group
law is given by
\[(s \tensor r^{-1})(q \tensor p^{-1})=s(r^{-1}q) \tensor p^{-1}=s \tensor (r^{-1}q)p^{-1}\]
where $(r^{-1}q) \in L$ denotes a value of a division map $\delta
\maps P \times_{M} P^{-1} \to L$  for the right action of $K$ on
$P$, which factors through the tensor product $P \tensor_{K} P^{-1}$
by an (unique) isomorphism
\[\bar{\delta} \maps P \tensor_{K} P^{-1} \to L\]
which shows that the sequence (\ref{regextact}) is exact and its
kernel is canonically isomorphic to $L$. The map $j_{P} \maps L \to
P \tensor_{K} G \tensor_{K} P^{-1}$ is defined for any point $l \in
L_{x}$ in the fiber over $x \in M$ by
\[ j_{P}(l)=(lq \tensor 1_{x} \tensor q^{-1})\] for any point $q \in
P_{x}$, and this map is independent of the choice of $q \in P_{x}$.
Now, if we denote the above action of $P$ on $G$ by $P \tri G$ then
for any $D-L$ bitorsor $Q$ and any $L-K$ bitorsor $P$ we have a
canonical isomorphism
\[\kappa_{Q,P,G} \maps (Q \tensor_{L} P) \tri G \to Q \tri (P \tri G)\]
given by the coherence associativity for the tensor product of
bitorsors.
\end{ex}

\section{Action bicategory}
\begin{thm} For any action (\ref{biaction}) of the bicategory $\B$ on the category $\Pe$,
there exists an {\it action bicategory} $\Pe \tril \B$ consisting of
the following data:
\begin{itemize}
\item Objects of $\Pe \tril \B$ are given by objects $P_{0}$ of the category $\Pe$

\item a 1-morphism is a pair $(\psi,h) \maps q \to p$
which we draw as an arrow
\[\xymatrix{q \ar[r]^-{(\psi,h)} & p }\] where $h \maps \Lambda_{0}(q) \to
\Lambda_{0}(p)$ is a 1-morphism in the bicategory $\B$, and $\psi
\maps q \to p \tril h$ is a morphism in the category $\Pe$, thus it
is an element of $P_{1}$.

\item a 2-morphism $\gamma \maps (\psi,h) \Rightarrow (\xi,l)$
\[\xymatrix@!=4pc{q \rtwocell<8>^{(\psi,h)}_{(\xi,l)}{\gamma} & p}\] is a 2-morphism $\gamma \maps h \Rightarrow
l$ in $B_{2}$, such that the diagram of morphisms in $\Pe$
\[\xymatrix@!=4pc{q \ar[r]^-{\psi} \ar[dr]_-{\xi} & p \tril h  \ar[d]^-{p \tril \gamma}\\
& p \tril l}\] commutes.
\end{itemize}
\begin{proof} We define the composition for any two
composable 1-morphisms
\[\xymatrix{r \ar[r]^-{(\phi,g)} & q \ar[r]^-{(\psi,h)} & p }\]  by
$(\psi,h) \circ (\phi,g)=(\psi \circ \phi,h \circ g) \maps r \to p,$
where $\psi \circ \phi \maps r \to p \tril (h \circ g)$ is a
morphism in $\Pe$, defined by the composition
\[\xymatrix{r \ar[r]^-{\phi} & q \tril g \ar[r]^-{\psi \tril g} &
(p \tril h) \tril g \ar[r]^-{\kappa_{p,h,g}} & p \tril (h \circ
g)}\] and we will show that this composition is a coherently
associative. For any three composable 1-morphisms
\[\xymatrix{s \ar[r]^-{(\varphi,f)} & r \ar[r]^-{(\phi,g)} & q \ar[r]^-{(\psi,h)} & p }\] first we have a morphism
$((\psi \circ \phi) \circ \varphi,(h \circ g) \circ f)$, where the
first term is a composite of
\[\xymatrix{s \ar[r]^-{\varphi} & r \tril f \ar[r]^-{(\psi \circ \phi) \tril f} &
(p  \tril (h \circ g)) \tril f \ar[r]^-{\kappa_{p,h \circ g,f}} & p
\tril ((h \circ g) \circ f)}\] Also we have the composition $(\psi
\circ (\phi \circ \varphi),h \circ (g \circ f)),$  and the first
term is given by a composite
\[\xymatrix{s \ar[r]^-{\phi \circ \varphi} & q \tril (g \circ f) \ar[r]^-{\psi \tril (g \circ f)} &
(p \tril h) \tril (g \circ f) \ar[r]^-{\kappa_{p,h,g \circ f}} & p
\tril (h \circ (g \circ f))}\] and the component of the
associativity $\alpha_{h,g,f} \maps (h \circ g) \circ f \to h \circ
(g \circ f)$, defines a 2-morphism
\[\xymatrix@!=4pc{s \rtwocell<8>^{((\psi \circ \phi) \circ \varphi,
(h \circ g) \circ f)}_{(\psi \circ (\phi \circ \varphi), h \circ (g
\circ f))}{\,\,\,\,\,\,\,\,\,\,\,\,\,\,\alpha_{h,g,f}} & p}\] which
we see from the commutativity of the diagram
\[\xymatrix{s \ar@{=}[dd] \ar[rr]^-{\varphi} && r \tril f
\ar[rr]^-{(\psi \circ \phi) \tril f} \ar[d]^-{\phi \tril f} && (p
\tril (h \circ g)) \tril f \ar[rr]^-{\kappa_{p,h \circ g,f}} &&
p \tril ((h \circ g) \circ f) \ar[dd]^-{p \tril \alpha_{h,g,f}}\\
&& (q \tril g) \tril f \ar[rr]^-{(\psi \tril g) \tril f}
\ar[d]^-{\kappa_{p,h,g}} &&
((p \tril h) \tril g) \tril f \ar[u]_-{\kappa_{p,h,g} \tril f} \ar[d]^-{\kappa_{p \tril h,g,f}} && \\
s \ar[rr]_-{\phi \circ \varphi} && q \tril (g \circ f)
\ar[rr]_-{\psi \tril (g \circ f)} && (p \tril h) \tril (g \circ
f)\ar[rr]_-{\kappa_{p,h,g \circ f}} && p \tril (h \circ (g \circ
f))}\] that follows from the definition of the horizontal
composition, the naturality and the coherence for quasiassociativity
of the action. The horizontal composition of 2-morphisms
\[\xymatrix@!=4pc{r \rtwocell<8>^{(\phi,g)}_{(\theta,k)}{\pi} & q \rtwocell<8>^{(\psi,h)}_{(\xi,l)}{\rho} & p}\]
is given by the horizontal composition in $B_{2}$
\[\xymatrix@!=4pc{r \rtwocell<8>^{(\psi \circ \phi,h \circ g)}_{(\xi \circ \theta,l \circ k)}{\,\,\,\,\,\,\,
\rho \circ \pi} & p}\] since we have a commutative diagram
\[\xymatrix{r \ar[dr]_-{\theta} \ar[rr]^-{\phi} && q \tril g \ar[rr]^-{\psi \tril g}
\ar[dr]_-{\xi \tril g} \ar[dl]^-{q \tril \pi} && (p \tril h) \tril g
\ar[dd]^-{(p \tril \rho) \tril \pi} \ar[dl]^-{(p \tril \rho) \tril
g} \ar[rr]^-{\kappa_{p,h,g}} &&
p \tril (h \circ g) \ar[dd]^-{p \tril (\rho \circ \pi)}\\
& q \tril k \ar[dr]_-{\xi \tril k} && (p \tril l) \tril g
\ar[dr]^-{(p \tril l) \tril \pi} \ar[dl]_-{(p \tril l) \tril \pi} &&& \\
&& (p \tril l) \tril k \ar@{=}[rr] && (p \tril l) \tril k
\ar[rr]_-{\kappa_{p,l,k}} && p \tril (l \circ k)}\] which follows
from the interchange law and the naturality of the coherence for the
quasiassociativity of the action. The vertical composition of
2-morphisms in $\Pe \tril \B$ is similarly induced from the one in
$\B$. The coherence of the horizontal composition in $\Pe \tril \B$
is immediately given by the coherence of the horizontal composition
in $\B$.
\end{proof}
\end{thm}

\begin{prop} There exists a canonical projection
\begin{equation}\label{homoproject}
\Lambda \maps \Pe \tril \B \to \B
\end{equation}
which is a strict homomorphism of bicategories.
\begin{proof} A homomorphism $\Lambda \maps \Pe \tril \B \to \B$ is defined by (the component of)
the momentum functor $\Lambda_{0}(p)=\lambda_{0}(p)$, for any object
$p$ in $\Pe \tril \B$. For any 1-morphism $(\psi,h)$ it is defined
by $\Lambda_{1}(\psi,h)=h$, and for any 2-morphism $\gamma \maps
(\psi,h) \Rightarrow (\xi,l)$ in $\Pe \tril \B$, it is given simply
by $\Lambda_{2}(\gamma)=\gamma$. Then we have a following identity
\[\Lambda((\psi,h) \circ (\phi,g))=\Lambda(\psi \circ \phi,h \circ
g)=h \circ g=\Lambda(\psi,h) \circ \Lambda(\phi,g)\] which means
that this homomorphism is strict (it preserves a composition
strictly).
\end{proof}
\end{prop}

\begin{ex} The right action of a bicategory $\B$ on itself is given by a diagram
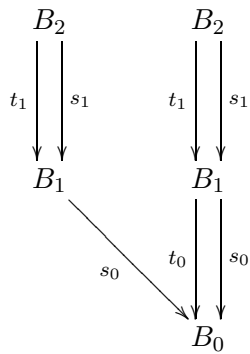
\[\xymatrix@!=3pc{B_{2} \ar@<1ex>[d]^-{s_{1}}
\ar@<-1ex>[d]_{t_{1}} & B_{2} \ar@<1ex>[d]^{s_{1}}
\ar@<-1ex>[d]_{t_{1}}\\ B_{1} \ar[dr]_{s_{0}} & B_{1}
\ar@<1ex>[d]^{s_{0}} \ar@<-1ex>[d]_{t_{0}}\\
& B_{0}}\] where a momentum functor is given by the source $S \maps
\B_{1} \to \B_{0}$ and an action functor is given by a horizontal
composition $H \maps \B_{1} \times_{\B_{0}} \B_{1} \to \B_{1}$. Any
object of an action bicategory $\B_{1} \tril \B$ is an element of
$B_{1}$, which which is a 1-morphism
\[\xymatrix{x \ar[r]^-{f} & y.}\] A 1-morphism from an object $f$ to
an object $f'$ is a pair $(\phi,g) \maps f \to f'$ as in the diagram
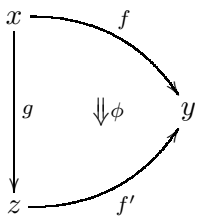
\[\xymatrix{x \ar[dd]^-{g} \ar@/^1pc/[drr]^{f} && \\
\rrtwocell<\omit>{\phi} && y\\
z \ar@/_1pc/[urr]_{f'} && }\] where $\phi \maps f \Rightarrow f'
\circ g$ is a 2-morphism in $\B$. A 2-morphism $\gamma \maps
(\phi,g) \Rightarrow (\psi,h)$ is a diagram
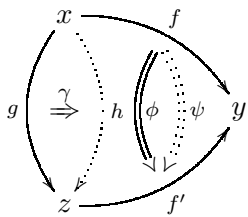
\[\xymatrix{x \ar@/_1.2pc/[dd]_{g} \ar@{.>}@/^1.2pc/[dd]^{h}
\ar@/^1pc/[drr]^{f}_{\ }="s" && \\
&& y\\
z \ar@/_1pc/[urr]_{f'}^{\ }="t" \uutwocell<\omit>{\gamma} &&
\ar@{=>}@/_.7pc/"s";"t"^{\phi} |>\hole
\ar@{:>}@/^.7pc/"s";"t"^{\psi} |>\hole}\] where $\gamma \maps g
\Rightarrow h$ is a 2-morphism in $\B$ such that identity $\psi=(f'
\circ \gamma)\phi$ holds. We will denote an action bicategory
$\B_{1} \tril \B$ by $T\B$, and we call it a {\it tangent
bicategory} because the 2-bundle
\begin{equation}\label{tang2bun}
T \maps T\B \to B_{0}
\end{equation}
(which associates to all above diagrams an object $y$) is a
generalization of a {\it tangent 2-bundle} introduced by Roberts and
Schreiber in \cite{RS} in the case of strict 2-categories. This
example of an action bicategory plays a crucial role in
understanding of universal 2-bundles.
\end{ex}

\section{Bigroupoid 2-torsors}

\begin{defi}
A right action of a bigroupoid $\B$ on a groupoid $\Pe$ is given by
the action of the underlying bicategory $\B$ on a category $\Pe$
given as previously by $(\Pe,\B,\Lambda,A,\alpha,\iota)$.
\end{defi}

\begin{defi} Let $\B$ be an internal bigroupoid in $\E$, and $\pi \maps \Pe \to X$
a right $\B$-2-bundle of groupoids over $X$ in $\E$. We say that
$(\Pe,\pi,\Lambda,A,X)$ is a right $\B$-principal-2-bundle (or a
right $\B$-torsor) over $X$ if the following conditions are
satisfied:
\begin{itemize}
\item the projection morphism $\pi_{0} \maps P_{0} \to X$
is an epimorphism,

\item the action morphism $\lambda_{0} \maps P_{0} \to
B_{0}$ is an epimorphism,

\item the induced internal functor
\begin{equation} \label{indeq}
\xymatrix{(Pr_{1},A) \maps \Pe \times_{B_{0}} \B_{1} \to \Pe
\times_{X} \Pe}
\end{equation}
is a (strong) equivalence of internal groupoids over $\Pe$ (where
both groupoids are seen as objects over $\Pe$ by the first
projection functor).

\end{itemize}
\end{defi}

\begin{ex} (The trivial 2-torsor) The trivial 2-torsor is given by the triple
$(\B_{1},T,S,\Ha,B_{0})$ where the momentum is given by the source
functor $S \maps \B_{1} \to \B_{0}$, and the action is given by the
horizontal composition $H \maps \B_{1} \times_{\B_{0}} \B_{1} \to
\B_{1}$.
\end{ex}

\begin{ex} For any $\B$-2-torsor $(\Pe,\pi,\Lambda,A,X)$ over $X$, and any morphism $f
\maps M \to B_{0}$, we have a pullback $\B$-2-torsor over $M$,
defined by $(f^{\ast}(\Pe),Pr_{1},\Lambda \circ
Pr_{2},f^{\ast}(A),X)$.
\end{ex}

\noindent Let us describe the simplicial set $\Pe_{\bullet}$ arising
by the application of the Duskin nerve functor
\[N_{2} \maps Bicat \to \Ss Set\]
to the action bicategory $\Pe \tril \B$. The set of 0-simplices is
$P_{0}$ and any 1-simplex is an arrow
\[\xymatrix@!=4pc{p_{j} \ar[r]^-{(\pi_{ij},f_{ij})} & p_{i}}\] and face operators are defined by
$d^{1}_{0}(\pi_{ij},f_{ij})=p_{i}$ and
$d^{1}_{1}(\pi_{ij},f_{ij})=p_{j}$, while the degeneracy is defined
by $s^{1}_{0}(p_{i})=(\iota_{p_{i}},i_{p_{i}})$ and it is given by
the arrow
\[\xymatrix@!=4pc{p_{i} \ar[r]^-{(\iota_{p_{i}},i_{p_{i}})} & p_{i}}\] where the morphism
$\iota_{p_{i}} \maps p_{i} \to p_{i} \tril i_{\Lambda_{0}(p_{i})}$
is an identity coherence of the action. A 2-simplex in
$\Pe_{\bullet}$ is of the form
\[\xymatrix@!=4pc{p_{k} \drtwocell<\omit>{<-4>\,\,\,\,\,\,\,\beta_{ijk}}
\ar[r]^-{(\pi_{jk},f_{jk})} \ar[dr]_-{(\pi_{ik},f_{ik})} & p_{j} \ar[d]^-{(\pi_{ij},f_{ij})}\\
& p_{i}}\] where the diagram
\[\xymatrix@!=4pc{p_{k} \ar[r]^-{\pi_{ij} \circ \pi_{jk}} \ar[dr]_-{\pi_{ik}} & p_{i} \tril (f_{ij} \circ f_{jk})
\ar[d]^-{p \tril \beta_{ijk}}\\ & p_{i} \tril f_{ik}}\] of morphisms
in $\Pe$ commutes, and the morphism $\pi_{ij} \circ \pi_{jk} \maps
p_{k} \to p_{i} \tril (f_{ij} \circ f_{jk})$ is the composite of
\[\xymatrix{p_{k} \ar[r]^-{\pi_{jk}} & p_{j} \tril f_{jk} \ar[r]^-{\pi_{ij} \tril f_{jk}} &
(p_{i} \tril f_{ij}) \tril f_{jk} \ar[r]^-{\kappa_{i,j,k}} & p_{i}
\tril (f_{ij} \circ f_{jk})}\] of morphisms in $\Pe$. Face operators
are defined by
\[\begin{array}{c}d^{2}_{0}(\beta_{ijk})=(\pi_{jk},f_{jk})\\
d^{2}_{1}(\beta_{ijk})=(\pi_{ik},f_{ik})\\
d^{2}_{2}(\beta_{ijk})=(\pi_{ij},f_{ij}) \end{array}\] and the
degeneracy operators are given by
\[\begin{array}{c} s^{2}_{0}(\pi_{ij},f_{ij})=\rho_{f_{ij}}\\
s^{2}_{1}(\pi_{ij},f_{ij})=\lambda_{f_{ij}}
\end{array}\] which are the two 2-simplices
\[\xymatrix@!=4pc{p_{j} \drtwocell<\omit>{<-4>\,\,\,\,\,\,\,\rho_{f_{ij}}}
\ar[r]^-{(\iota_{p_{j}},i_{p_{j}})} \ar[dr]_-{(\pi_{ij},f_{ij})} & p_{j} \ar[d]^-{(\pi_{ij},f_{ij})}\\
& p_{i}} \hspace{3pc} \xymatrix@!=4pc{p_{j}
\drtwocell<\omit>{<-4>\,\,\,\,\,\,\,\lambda_{f_{ij}}}
\ar[r]^-{(\pi_{ij},f_{ij})} \ar[dr]_-{(\pi_{ij},f_{ij})} & p_{i} \ar[d]^-{(\iota_{p_{i}},i_{p_{i}})}\\
& p_{i}}\] respectively, where the 1-morphisms $\rho_{f_{ij}} \maps
f_{ij} \circ i_{p_{j}} \to f_{ij}$ and $\lambda_{f_{ij}} \maps
i_{p_{i}} \circ f_{ij} \to f_{ij}$ are the components of the right
and left identity natural isomorphisms in $\B$. \\A general
3-simplex is of the form
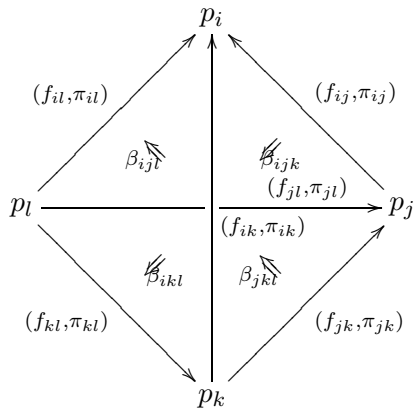
\[\xymatrix@!=4pc{& p_{i}  \drtwocell<\omit>{<4>\beta_{ijk}}  \dltwocell<\omit>{<-4>\beta_{ijl}} &\\
p_{l} \drtwocell<\omit>{<-4>\beta_{ikl}} \ar[ur]^{(f_{il},\pi_{il})}
\ar'[r][rr]^{(f_{jl},\pi_{jl})} \ar[dr]_-{(f_{kl},\pi_{kl})}  &&
p_{j} \dltwocell<\omit>{<4>\beta_{jkl}} \ar[ul]_-{(f_{ij},\pi_{ij})} \\
& p_{k} \ar[uu]_(0.45){(f_{ik},\pi_{ik})}
\ar[ur]_-{(f_{jk},\pi_{jk})} &}\] where we have an identity
\[\beta_{ikl}(\beta_{ijk} \circ f_{kl})= \alpha_{ijkl}\beta_{ijl}(\beta_{jkl} \circ f_{ij})\]
which is just a nonabelian 2-cocycle condition.

\begin{ex} Let $B_{\bullet}$ be Duskin nerve for a bicategory $\B$.
The tangent bicategory $T\B$ from Example 6.1. is action bicategory
for the right action of $\B$ on itself and a d\'ecalage construction
(\ref{Decalage}) from Chapter 2 becomes the diagram of simplicial
sets
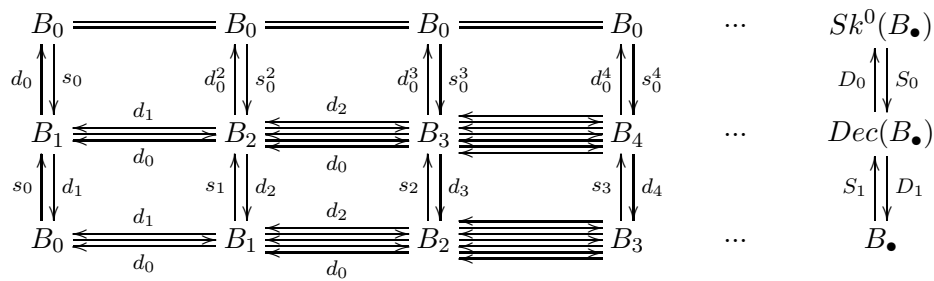
\[\xymatrix{B_{0} \ar@{=}[rr] \ar@<0.5ex>[d]^{s_{0}} && B_{0} \ar@{=}[rr] \ar@<0.5ex>[d]^{s^{2}_{0}}
&& B_{0} \ar@{=}[rr] \ar@<0.5ex>[d]^{s^{3}_{0}} &&
B_{0} \ar@<0.5ex>[d]^{s^{4}_{0}}&...& Sk^{0}(B_{\bullet}) \ar@<0.5ex>[d]^{S_{0}}\\
B_{1} \ar@<0.5ex>[u]^{d_{0}} \ar[rr] \ar@<0.5ex>[d]^{d_{1}} && B_{2}
\ar@<0.5ex>[u]^{d^{2}_{0}} \ar@<-0.5ex>[ll]_{d_{1}}
\ar@<0.5ex>[ll]^{d_{0}} \ar@<-0.5ex>[rr] \ar@<0.5ex>[rr]
\ar@<0.5ex>[d]^{d_{2}} && B_{3} \ar@<0.5ex>[u]^{d^{3}_{0}} \ar[ll]
\ar@<-1ex>[ll]_{d_{2}} \ar@<1ex>[ll]^{d_{0}} \ar@<0.5ex>[d]^{d_{3}}
\ar[rr] \ar@<-1ex>[rr] \ar@<1ex>[rr] && B_{4}
\ar@<0.5ex>[u]^{d^{4}_{0}} \ar@<0.5ex>[d]^{d_{4}} \ar@<-0.5ex>[ll]
\ar@<0.5ex>[ll] \ar@<-1.5ex>[ll] \ar@<1.5ex>[ll]
&...& Dec(B_{\bullet}) \ar@<0.5ex>[d]^{D_{1}} \ar@<0.5ex>[u]^{D_{0}} \\
B_{0} \ar@<0.5ex>[u]^{s_{0}} \ar[rr]  && B_{1}
\ar@<0.5ex>[u]^{s_{1}} \ar@<-0.5ex>[ll]_{d_{1}}
\ar@<0.5ex>[ll]^{d_{0}} \ar@<-0.5ex>[rr] \ar@<0.5ex>[rr] && B_{2}
\ar@<0.5ex>[u]^{s_{2}} \ar[ll] \ar@<-1ex>[ll]_{d_{2}}
\ar@<1ex>[ll]^{d_{0}} \ar[rr] \ar@<-1ex>[rr] \ar@<1ex>[rr] && B_{3}
\ar@<0.5ex>[u]^{s_{3}} \ar@<-0.5ex>[ll] \ar@<0.5ex>[ll]
\ar@<-1.5ex>[ll] \ar@<1.5ex>[ll] &...& B_{\bullet}
\ar@<0.5ex>[u]^{S_{1}}}\] in which $D_{1} \maps Dec(B_{\bullet}) \to
B_{\bullet}$ is a simplicial map which is the Duskin nerve of the
canonical projection $\Lambda \maps T\B \to \B$ and $D_{0} \maps
Dec(B_{\bullet}) \to B_{\bullet}$ is a simplicial map which is the
Duskin nerve of the tangent 2-bundle $T \maps T\B \to B_{0}$
\end{ex}

\begin{thm} Let the bigroupoid $\B$ acts on a groupoid $\Pe$.
Then the Duskin nerve of the canonical projection
(\ref{homoproject}) is a simplicial map
$\Lambda_{\bullet}=\N_{2}(\Lambda) \maps \Pe_{\bullet} \to
\B_{\bullet}$ which is a simplicial action of the Duskin nerve
$B_{\bullet}$ on the bigroupoid $\B$, i.e. it is an exact fibration
for all $n \geq 2$.
\begin{proof} We need to show that for any $n \geq 2$ and for any $k$ such that $0 \leq k \leq n$, the diagram
\[\xymatrix@!=4pc{P_{n} \ar[d]_{p_{\bar{k}}} \ar[r]^-{\lambda_{n}} & B_{n} \ar[d]^{p_{\bar{k}}} \\
\bigwedge^{k}_{n}(\Pe_{\bullet}) \ar[r]_{\lambda^{k}_{n}} &
\bigwedge^{k}_{n}(\B_{\bullet})}\] is a pullback. A k-horn
$((f_{ij},\pi_{ij}),...,(f_{j,k-1},\pi_{j,k-1}),(f_{k,k+1},\pi_{k,k+1}),...,(f_{n-1,n},\pi_{n-1,n}))$
in $\bigwedge^{k}_{n}(\Pe_{\bullet})$ is given by the n-tuple of
1-morphisms in $\A_{\B}\Pe$, and its image by $\lambda^{k}_{2} \maps
\bigwedge^{k}_{2}(\Pe_{\bullet})  \to
\bigwedge^{k}_{2}(\Pe_{\bullet})$ is a k-horn in
$\bigwedge^{k}_{n}(\B_{\bullet})$, given by the n-tuple
$(f_{ij},...,f_{j,k-1},f_{k,k+1},...,f_{n-1,n})$ of 1-morphisms in
$\B$. For example, in the case $n=2$, any filler of a 1-horn
$(f_{ij},-,f_{jk})$ in $\bigwedge^{1}_{2}(\B_{\bullet})$, is the
2-simplex
\[\xymatrix@!=4pc{x_{k} \drtwocell<\omit>{<-4>\,\,\,\,\,\,\,\beta_{ijk}}
\ar[r]^-{f_{jk}} \ar[dr]_-{f_{ik}} & x_{j} \ar[d]^-{f_{ij}}\\&
x_{i}}\] in $B_{2}$. A 2-simplex in $\Pe_{\bullet}$ is a lifting of
the previous 2-simplex if it is of the form
\[\xymatrix@!=4pc{p_{k} \drtwocell<\omit>{<-4>\,\,\,\,\,\,\,\beta_{ijk}}
\ar[r]^-{(\pi_{jk},f_{jk})} \ar[dr]_-{(\pi_{ik},f_{ik})} & p_{j} \ar[d]^-{(\pi_{ij},f_{ij})}\\
& p_{i}}\] where the diagram
\[\xymatrix@!=4pc{p_{k} \ar[r]^-{\pi_{ij} \circ \pi_{jk}} \ar[dr]_-{\pi_{ik}} & p_{i} \tril (f_{ij} \circ f_{jk})
\ar[d]^-{p \tril \beta_{ijk}}\\ & p_{i} \tril f_{ik}}\] of morphisms
in $\Pe$ commutes, and the morphism $\pi_{ij} \circ \pi_{jk} \maps
p_{k} \to p_{i} \tril (f_{ij} \circ f_{jk})$ is the composite of
\[\xymatrix{p_{k} \ar[r]^-{\pi_{jk}} & p_{j} \tril f_{jk} \ar[r]^-{\pi_{ij} \tril f_{jk}} &
(p_{i} \tril f_{ij}) \tril f_{jk} \ar[r]^-{\kappa_{i,j,k}} & p_{i}
\tril (f_{ij} \circ f_{jk})}\] so we see that a pair
$((f_{ij},\pi_{ij}),-,(f_{jk},\pi_{jk}),\beta_{ijk})$ in
$\bigwedge^{1}_{2}(\Pe_{\bullet})
\times_{\bigwedge^{1}_{2}(\B_{\bullet})} B_{2}$ uniquely determines
above 2-simplex in $\Pe_{2}$. Since $\Pe$ is a groupoid, any pair
consisting of a k-horn in $\bigwedge^{k}_{2}(\B_{\bullet})$, for
$k=0,2$, and a 2-simplex in $\B_{2}$ which covers the k-horn,
uniquely determines a 2-simplex in $\Pe_{2}$, and thus provides a
canonical isomorphism $P_{2} \simeq \bigwedge^{k}_{2}(\Pe_{\bullet})
\times_{\bigwedge^{k}_{2}(\B_{\bullet})} B_{2}$. Since both
simplicial objects are 2-coskeletal, the assertion follows for all
$n \geq 2$.
\end{proof}
\end{thm}

\begin{defi} An action of the n-dimensional Kan complex is an internal simplicial
map $\Lambda_{\bullet} \maps \Pe_{\bullet} \to \B_{\bullet}$ in $\E$
which is a weak exact fibration for all $m \geq n$.
\end{defi}

In the case of the bigroupoid $\B$, the Duskin nerve functor is a
2-dimensional hypergroupoid $\B_{\bullet}=\N_{2}(\B)$ and let
$\Pe_{\bullet}=\N_{2}(\A_{\B}\Pe)$ be the Duskin nerve of an action
bigroupoid associated to the action of the bigroupoid $\B$ on the
groupoid $\Pe$. Glenn introduced in \cite{Gl} a simplicial
definition of an n-dimensional hypergroupoid n-torsor in $\E$.

\begin{defi} An action $\Lambda_{\bullet} \maps P_{\bullet}
\to \B_{\bullet}$ is the n-dimensional hypergroupoid n-torsor over
$X$ in $\E$ if $P_{\bullet}$ is augmented over $X$, aspherical and
n-1-coskeletal ($P_{\bullet} \simeq Cosk^{n-1}(P_{\bullet}))$.
\end{defi}

In the case of the bigroupoid $\B$, the above definition reduces to
the following one.

\begin{defi} A bigroupoid $\B_{\bullet}$ 2-torsor over an object $X$ in $\E$ is
an internal simplicial map $\Lambda_{\bullet} \maps P_{\bullet} \to
\B_{\bullet}$ in $\Ss(\E)$, which is an exact fibration for all $n
\geq 2$, and where $P_{\bullet}$ is augmented over $X$, aspherical
and 1-coskeletal ($P_{\bullet} \simeq Cosk^{1}(P_{\bullet}))$.
\end{defi}

\noindent Thus in the case when an action of $\B$ on $\Pe$ is
principal, we have the main result of our paper.

\begin{thm} Let $\Pe$ be a $\B$-2-torsor over $X$.
Then simplicial map $\Lambda_{\bullet}=\N_{2}(\Lambda) \maps
\Pe_{\bullet} \to \B_{\bullet}$ is a Duskin-Glenn 2-torsor.
\begin{proof} The simplicial complex $\Pe_{\bullet}$ is augmented
over $X$ because the action of $\B$ is fiberwise, since for any
1-simplex $(f_{ij},\pi_{ij}) \maps p_{j} \to p_{i}$ in $P_{0}$,
where $\pi_{ij} \maps p_{j} \to p_{i} \tril f_{ij}$ we have
\[\pi_{0} d_{0}(f_{ij},\pi_{ij})=\pi_{0}(p_{i})=\pi_{0}(p_{i} \tril
f_{ij})=\pi_{1}(\pi_{ij})=\pi_{0}(p_{j})=\pi_{0}
d_{1}(f_{ij},\pi_{ij}).\] The simplicial complex $\Pe_{\bullet}$ is
obviously aspherical and we prove now that it is also 1-coskeletal.
A general 2-simplex in $Cosk^{1}(P_{\bullet})_{2}$ is a triple
$((f_{ij},\pi_{ij}),(f_{ik},\pi_{ik}),(f_{jk},\pi_{jk}))$ which we
see as the triangle
\[\xymatrix@!=4pc{p_{k} \ar[r]^-{(\pi_{jk},f_{jk})} \ar[dr]_-{(\pi_{ik},f_{ik})} &
p_{j} \ar[d]^-{(\pi_{ij},f_{ij})}\\& p_{i}}\] from which we have
morphisms $\pi_{ij} \circ \pi_{jk} \maps p_{k} \to p_{i} \tril
(f_{ij} \circ f_{jk})$ and $\pi_{ik} \maps p_{k} \to p_{i} \tril
f_{ik}$ in $\Pe$. Now we use the fact that the induced functor
\[\xymatrix{(Pr_{1},\A) \maps \Pe \times_{B_{0}} \B_{1} \ar[r] & \Pe \times_{X} \Pe}\]
is a (strong) equivalence of internal groupoids over $\Pe$, and
therefore fully faithful. Specially, for the two objects
$(p_{i},f_{ij} \circ f_{jk})$ and $(p_{i},f_{ik})$ of $\Pe
\times_{B_{0}} \B_{1}$, this equivalence induces a bijection
\[\xymatrix{Hom_{\Pe \times_{B_{0}} \B_{1}}((p_{i},f_{ij} \circ f_{jk}),(p_{i},f_{ik})) \simeq
Hom_{\Pe \times_{X} \Pe}((p_{i},p_{i} \tril (f_{ij} \circ
f_{jk})),(p_{i},p_{i} \tril f_{ik}))}\] and therefore for a morphism
$(id_{p_{i}},\pi_{ik} \circ (\pi_{ij} \circ \pi_{jk})^{-1}) \maps
(p_{i},p_{i} \tril (f_{ij} \circ f_{jk})) \to (p_{i},p_{i} \tril
f_{ik}))$
\[\xymatrix@!=4pc{p_{k} \ar[dr]_-{\pi_{ik}} & p_{i} \tril (f_{ij} \circ f_{jk})
\ar[l]_-{(\pi_{ij} \circ \pi_{jk})^{-1}}\\& p_{i} \tril f_{ik}}\]
there exists a unique 2-morphism $\beta_{ijk} \maps f_{ij} \circ
f_{jk} \to f_{ik}$ in $\B$, such that the diagram
\[\xymatrix@!=4pc{p_{k} \ar[r]^-{\pi_{ij} \circ \pi_{jk}} \ar[dr]_-{\pi_{ik}} & p_{i} \tril (f_{ij} \circ f_{jk})
\ar[d]^-{p \tril \beta_{ijk}}\\ & p_{i} \tril f_{ik}}\] commutes,
and this uniquely determines a 2-simplex
\[\xymatrix@!=4pc{p_{k} \drtwocell<\omit>{<-4>\,\,\,\,\,\,\,\beta_{ijk}}
\ar[r]^-{(\pi_{jk},f_{jk})} \ar[dr]_-{(\pi_{ik},f_{ik})} & p_{j} \ar[d]^-{(\pi_{ij},f_{ij})}\\
& p_{i}}\] in $\Pe_{2}$, which proves that we have a bijection
$\Pe_{2} \simeq Cosk^{1}(P_{\bullet})_{2}$. From here it follows
immediately that $\Pe_{\bullet} \simeq Cosk^{1}(P_{\bullet})$.
\end{proof}
\end{thm}

\end{document}